\documentclass[a4paper,11pt]{amsart}
\input{xypic}
\newtheorem{proposition}{Proposition}[section]
\newtheorem{lemma}[proposition]{Lemma}
\newtheorem{corollary}[proposition]{Corollary}
\newtheorem{theorem}[proposition]{Theorem}

\theoremstyle{definition}
\newtheorem{definition}[proposition]{Definition}
\newtheorem{example}[proposition]{Example}

\theoremstyle{remark}
\newtheorem{remark}[proposition]{Remark}

\newcommand{\thlabel}[1]{\label{th:#1}}
\newcommand{\thref}[1]{Theorem~\ref{th:#1}}
\newcommand{\selabel}[1]{\label{se:#1}}
\newcommand{\seref}[1]{Section~\ref{se:#1}}
\newcommand{\lelabel}[1]{\label{le:#1}}
\newcommand{\leref}[1]{Lemma~\ref{le:#1}}
\newcommand{\prlabel}[1]{\label{pr:#1}}
\newcommand{\prref}[1]{Proposition~\ref{pr:#1}}
\newcommand{\colabel}[1]{\label{co:#1}}
\newcommand{\coref}[1]{Corollary~\ref{co:#1}}

\newcommand{\relabel}[1]{\label{re:#1}}
\newcommand{\reref}[1]{Remark~\ref{re:#1}}
\newcommand{\exlabel}[1]{\label{ex:#1}}
\newcommand{\exref}[1]{Example~\ref{ex:#1}}
\newcommand{\delabel}[1]{\label{de:#1}}

\newcommand{\eqlabel}[1]{\label{eq:#1}}
\newcommand{\equref}[1]{(\ref{eq:#1})}

\def\equal#1{\smash{\mathop{=}\limits^{#1}}}

\newcommand{\can}{{\rm can}}

\newcommand{\Pic}{{\rm Pic}}

\newcommand{\Hom}{{\rm Hom}}
\newcommand{\HOM}{{\rm HOM}}
\newcommand{\End}{{\rm End}}
\newcommand{\Alg}{{\rm Alg}}
\newcommand{\END}{{\rm END}}

\newcommand{\Ker}{{\rm Ker}\,}

\newcommand{\im}{{\rm Im}\,}

\newcommand{\Gal}{{\rm Gal}}

\def\lan{\langle}
\def\ran{\rangle}
\def\ot{\otimes}

\def\sq{\square}

\def\CC{{\mathbb C}}

\def\QQ{{\mathbb Q}}

\newcommand{\Cc}{\mathcal{C}}
\newcommand{\Dd}{\mathcal{D}}

\newcommand{\Mm}{\mathcal{M}}

\def\*C{{}^*\hspace*{-1pt}{\Cc}}
\def\text#1{{\rm {\rm #1}}}

\def\Morita{\dul{\rm Morita}}

\def\ol{\overline}
\def\ul{\underline}
\def\dul#1{\underline{\underline{#1}}}

\def\mapright#1{\smash{\mathop{\longrightarrow}\limits^{#1}}}

\begin{document}

\title[$H$-Picard groups]{Hopf-Galois extensions and an exact sequence for $H$-Picard groups}

\author[S. Caenepeel]{Stefaan Caenepeel}
\address{Faculty of Engineering,
Vrije Universiteit Brussel, Pleinlaan 2, B-1050 Brussels, Belgium}
\email{scaenepe@vub.ac.be}
\urladdr{http://homepages.vub.ac.be/\~{}scaenepe/}

\author[A. Marcus]{Andrei Marcus}
\address{Faculty of Mathematics and Computer Science,
Babe\c s-Bolyai University, Str. Mihail Kog\u alniceanu 1,
RO-400084 Cluj-Napoca, Romania} \email{marcus@math.ubbcluj.ro}
\urladdr{http://math.ubbcluj.ro\~{}marcus/}


\begin{abstract}
Let $H$ be a Hopf algebra, and $A$ an $H$-Galois extension. We
investigate $H$-Morita autoequivalences of $A$, introduce the
concept of $H$-Picard group, and we establish an exact sequence
linking the $H$-Picard group of $A$ and the Picard group of $A^{{\rm
co}H}$.
\end{abstract}

\thanks{This research was supported by the bilateral project BWS04/04 ``New
Techniques in Hopf algebras and graded ring theory" of the Flemish
and Romanian governments and by the research project G.0622.06
``Deformation quantization methods for algebras and categories
with applications to quantum mechanics" from FWO-Vlaanderen. The
second  author acknowledges the support of a Bolyai Fellowship of
the Hungarian Academy of Science and of the Romanian
PN-II-IDEI-PCE-2007-1 project,  code   ID\_532, contract no.
29/28.09.2007}

\maketitle

\section{Introduction}

The aim of this paper is the following generalization, presented
in \seref{6} below, of the main result of M. Beattie and A. del
R\'\i o \cite{BR2} (see also \cite{Marcus2} for an approach based
on \cite{Marcus}).

\begin{theorem}\thlabel{t:main}  Assume that $H$ is a cocommutative Hopf algebra over the field $k$.
Let $A$ be a faithfully flat $H$-Galois
extension. There is an exact sequence
$$1\to H^1(H, Z(A^{{\rm co}H}))\overset{g_1}\to\Pic^H(A)
\overset{g_2}\to\Pic(A^{{\rm co}H})^H \overset{g_3}\to H^2(H,
Z(A^{{\rm co}H})).$$
\end{theorem}

Here $H^*(H, Z(A^{{\rm co}H}))$ are the Sweedler cohomology groups
(with respect to the Miyashita-Ulbrich action of $H$ on $Z(A^{{\rm
co}H})$), $\Pic(A^{{\rm co}H})^H$ is the group of $H$-invariant
elements of $\Pic(A^{{\rm co}H})$ and $\Pic^H(A)$ is the group of
isomorphism classes of invertible relative Hopf bimodules. We
shall give later more details about these notations. Moreover,
$g_1$ and $g_2$ are group-homomorphisms, while $g_3$ is not.

We give a  proof of the theorem by using the ideas of
\cite{Marcus2} and the  results of \cite{CCMT} and \cite{MiSt},
obtaining in this way an interesting interpretation of the above
theorem in terms of Clifford extendibility to $A$ of $A^{{\rm
co}H}$-modules.

The paper is divided as follows. In \seref{1} we present our
general setting, which involves Hopf-Galois extensions, the
Miyashita-Ulbrich action, and most importantly, the concepts of
$H$-Morita context and $\square_H$-Morita context introduced in
\cite{CCMT}, and their relationship with Hopf subalgebras. The
main result of \seref{2} says that if $H$ is cocommutative and $A$
is a faithfully flat $H$-Galois extension of $B:=A^{{\rm co}H}$,
then the cotensor product $A^{\sq e}:=A\sq_H A^{\rm op}$ is a
faithfully flat Hopf-Galois extension of the enveloping algebra
$B^e:=B\ot B^{\rm op}$. In the first part of \seref{4} we discuss the particular
case when $A$ is a cleft extension of the commutative algebra
$B:=A^{{\rm co}H}$, and especially, the characterization of this
situation in terms in Sweedler's $1$- and $2$-cohomology. This is
needed in the second part of \seref{4}, where we review and adapt to our needs
the results  of Militaru and \c Stefan \cite{MiSt} on Clifford
extendibility of modules. The cleft extension in discussion is the
subalgebra $E:={}_A\END(A\ot_B M)^{\rm op}$ of rational elements
in ${}_A\End(A\ot_B M)^{\rm op}$, where $M$ is an $H$-invariant
$B$-module, and $E^{{\rm co}H}\simeq {}_B\End(M)^{\rm op}$ is
assumed to be commutative. In \seref{5} we introduce the
$H$-Picard group $\Pic^H(A)$ and the
$\square_H$-Picard group ${\Pic}^{\sq_H}(A^{{\rm co}H})$ of
$A^{{\rm co}H}$. It is a consequence of the results of \cite{CCMT}
that the groups ${\Pic}^H(A)$ and ${\Pic}^{\sq_H}(A^{{\rm co}H})$
are isomorphic. In the situation where $H$ is cocommutative, we can
introduce the subgroup $\Pic(A^{{\rm co}H})^H$ of $\Pic(A^{{\rm co}H})$
consisting of $H$-stable elements of $\Pic(A^{{\rm co}H})$ (\seref{5b}).
 The definitions of the maps $g_1$, $g_2$ and
$g_3$, as well as the proof of the main theorem are given in
\seref{6}. The main ingredient here is the application of the
Militaru-Stefan lifting theorem to an $H$-stable invertible
$(B,B)$-bimodule $M$, by considering the cleft extension
$E:={}_{A^{\sq e}}\END(A^{\sq e}\ot_{B^e}M)^{\rm op}$ of $E^{{\rm
co}H}\cong Z(B)$. Note that the action of $H$ on $Z(B)$ coming
from $E$ is the same as the Miyashita-Ulbrich action  coming from
$A$, hence it is independent of $M$. \seref{7} is
concerned with the analysis of the map $g_3$. It turns out that
the action  $\Pic(B)$ on $Z(B)$ induces an action of $\Pic(B)^H$
on $H^n(H,Z(B))$, and  that $g_3$ is an $1$-cocycle of the group
$\Pic(B)^H$ with values in $H^2(H,Z(B))$.\\
The exact sequence describing $\Pic^H(A)$ given in \seref{6} holds
in the case where $H$ is cocommutative; in the general case, we can
still give a description of $\Pic^H(A)$, in the case where the coinvariants
of $A$ coincide with the groundfield, that is, $A$ is an $H$-Galois object.
This is done in \seref{9}, and involves Schauenburg's theory of
bigalois objects.\\
Modules will be unital and left, unless otherwise stated. For
general results on Hopf algebras the reader is referred  to
\cite{CMZ}, \cite{DNR} or \cite{Mont}. For group graded versions
of the topics discussed here we also mention  \cite{BR1} and
\cite{HR}.

\section{Hopf-Galois extensions}\selabel{1}
Throughout this paper, $H$ is a Hopf algebra, with bijective
antipode $S$, over a field $k$. We use the Sweedler notation for
the comultiplication on $H$: $\Delta(h)= h_{(1)}\ot h_{(2)}$.
$\Mm^H$ (respectively ${}^H\Mm$) is the category of right
(respectively left) $H$-comodules. For a right $H$-coaction $\rho$
(respectively a left $H$-coaction $\lambda$) on a $k$-module $M$,
we denote
$$\rho(m)=m_{[0]}\ot m_{[1]}\quad{\rm and}\quad
\lambda(m)=m_{[-1]}\ot m_{[0]}.$$ The submodule of coinvariants
$M^{{\rm co}H}$ of a right (respectively left) $H$-co\-mo\-du\-le
$M$ consists of the elements $m\in M$ satisfying $\rho(m)=m\ot 1$
(respectively
$\lambda(m)=1\ot m$).\\
Let $A$ be a right $H$-comodule algebra. ${}_A\Mm^H$ and $\Mm^H_A$
are the categories of left and right relative Hopf modules, and
${}_A\Mm^H_A$ is the category of relative Hopf bimodules, see \cite{CCMT}.
$B=A^{{\rm co}H}$ will be the subalgebra of coinvariants of $A$.
We have two pairs of adjoint functors $(F_1 =A\ot_{B}-,\
G_1=(-)^{{\rm co}H})$ and $(F_2=-\ot_{B}A,\ G_2=(-)^{{\rm co}H})$
between the categories ${}_{B}\Mm$ and ${}_A\Mm^H$, and between
$\Mm_{B}$ and $\Mm^H_A$. Consider the canonical maps
\begin{align*} \can:\ A&\ot_{B} A\to A\ot H,\quad \can(a\ot_B
b)=ab_{[0]}\ot b_{[1]}; \\
\can':\ A&\ot_{B} A\to A\ot H,\quad  \can'(a\ot_B b)=a_{[0]}b\ot
a_{[1]}.\end{align*} We have the following result, due to H.-J.
Schneider \cite[Theorem I]{Schneider0}.

\begin{theorem}\thlabel{1.1}
For a right $H$-comodule algebra $A$, the following statements are equivalent.
\begin{enumerate}
\item $(F_2,G_2)$ is a pair of inverse equivalences;
\item $(F_2,G_2)$ is a pair of inverse equivalences and $A\in {}_{B}\Mm$ is flat;
\item $\can$ is an isomorphism and $A\in {}_{B}\Mm$ is faithfully flat;
\item $(F_1,G_1)$ is a pair of inverse equivalences;
\item $(F_1,G_1)$ is a pair of inverse equivalences and $A\in \Mm_{B}$ is flat;
\item $\can'$ is an isomorphism and $A\in \Mm_{B}$ is faithfully flat.
\end{enumerate}
If these conditions are satisfied, then we say that $A$ is a faithfully flat $H$-Galois extension
of $B$.
\end{theorem}

\subsection*{The Miyashita-Ulbrich action} Let $A$ be a
faithfully flat right $H$-Galois extension, and consider the map
$$\gamma_A=\can^{-1}\circ(\eta_A\ot H):\ H\to A\ot_{B}A,\quad h\mapsto
\sum_i l_i(h)\ot_{B}r_i(h).$$ Then the element $\gamma_A(h)$ is
characterized by the property
\begin{equation}\eqlabel{1.2.1}
\sum_i l_i(h)r_i(h)_{[0]}\ot r_i(h)_{[1]}=1\ot h.
\end{equation}
For all $h,h'\in H$ and $a\in A$, we have (see \cite[3.4]{Schneider1}):
\begin{eqnarray}
&&\gamma_A(h)\in (A\ot_{B}A)^{B};\eqlabel{1.2.2}\\
&&\gamma_A(h_{(1)})\ot h_{(2)}=
\sum_i l_i(h)\ot_{B} r_i(h)_{[0]}\ot r_i(h)_{[1]};\eqlabel{1.2.3}\\
&&\gamma_A(h_{(2)})\ot S(h_{(1)})=
\sum_i l_i(h)_{[0]}\ot_{B} r_i(h)\ot l_i(h)_{[1]};\eqlabel{1.2.4}\\
&&\sum_i l_i(h)r_i(h)=\varepsilon(h)1_A;\eqlabel{1.2.5}\\
&&\sum_i a_{[0]}l_i(a_{[1]})\ot_B r_i(a_{[1]})=1\ot_B a;\eqlabel{1.2.6}\\
&&\gamma_A(hh')=\sum_{i,j} l_i(h')l_j(h)\ot_{B}
r_j(h)r_i(h').\eqlabel{1.2.7}
\end{eqnarray}
Using the above formulas, it is straightforward to show that $Z(B)$, the center
of $B$, is a right $H$-module algebra under the Miyashita-Ulbrich action:
$$x\bullet h=\sum_i l_i(h)xr_i(h),$$
for all $x\in Z(B)$, $h\in H$. In what follows, we will view
$Z(B)$ as a left $H$-module algebra via
\begin{equation}\eqlabel{1.2.8}
h\cdot x=x\bullet S^{-1}(h)=\sum_i l_i(S^{-1}(h))xr_i(S^{-1}(h)).
\end{equation}
We will need the following commutation rule in the sequel.

\begin{lemma}\lelabel{1.2}
For $x\in Z(B)$ and $a\in A$, we have
\begin{equation}\eqlabel{1.2.9}
xa=a_{[0]}(S(a_{[1]})\cdot x)\quad{\rm and}\quad
ax=(a_{[1]}\cdot x)a_{[0]}.
\end{equation}
\end{lemma}

\begin{proof}
From \equref{1.2.6}, we know that $\sum_i a_{[0]}l_i(a_{[1]})\ot_B r_i(a_{[1]})=1\ot_B a\in B\ot_B A
\subset A\ot_B A$,
and then we can see that
$$x\ot_B a=\sum_i xa_{[0]}l_i(a_{[1]})\ot_B r_i(a_{[1]})=
\sum_i a_{[0]}l_i(a_{[1]})x\ot_B r_i(a_{[1]}),$$
hence
$$xa= \sum_i a_{[0]}l_i(a_{[1]})x r_i(a_{[1]})=a_{[0]}(S(a_{[1]})\cdot x).$$
For all $h\in H$, we have that $h\cdot x\in Z(B)$. Apply the first formula of \equref{1.2.9}
with $x$ replaced by $a_{[1]}\cdot x$; this gives the second formula:
$$(a_{[1]}\cdot x)a_{[0]}=a_{[0]}((S(a_{[1]})a_{[2]})\cdot x)=ax.$$
\end{proof}

\subsection*{Morita equivalences} We recall here some concepts and
results from \cite{CCMT}. These are the main ingredients in the
definition of $\Pic^H(A)$ and of the maps $g_1$ and $g_2$ in
\thref{t:main}.
\begin{definition}\delabel{1.4}
Let $A$ and $A'$ be right $H$-comodule algebras. An $H$-Morita context
connecting $A$ and $A'$ is a Morita context $(A,A',M,N,\alpha,\beta)$ such that
$M\in {}_A\Mm_{A'}^H$, $N\in {}_{A'}\Mm_A^H$,  $\alpha:\ M\ot_{A'} N\to A$
is a morphism in ${}_A\Mm_A^H$ and $\beta:\ N\ot_A M\to A'$ is a morphism in
${}_{A'}\Mm_{A'}^H$.
\end{definition}

\begin{definition}\delabel{1.5}
Assume that $A$ and $A'$ are right faithfully flat $H$-Galois
extensions of $A^{{\rm co}H}=B$ and ${A'}^{{\rm co}H}=B'$. A $\sq_H$-{\it
Morita context} between $B$ and $B'$ is a
Morita context $(B,B',M_1,N_1,\alpha_1,\beta_1)$ such that $M_1$ (resp. $N_1$) is a
left $A\sq_H{A'}^{\rm op}$-module (resp. $A'\sq_HA^{\rm op}$-module) and
\begin{itemize}
\item $\alpha_1:\ M_1\ot_{B'} N_1\to B$ is left $A\sq_HA^{\rm op}$-linear,
\item $\beta_1:\ N_1\ot_{B} M_1\to B'$ is left $A'\sq_H{A'}^{\rm op}$-linear.
\end{itemize}
\end{definition}

$\Morita(B,B')$ is the category with Morita contexts connecting
$B$ and $B'$ as objects. A morphism between the Morita contexts
$(B,B',M_1,N_1,\alpha_1,\beta_1)$ and
$(B,B',M_2,N_2,\alpha_2,\beta_2)$ is a couple $(\mu,\nu)$, with
$\mu:\ M_1\to M_2$ and $\nu:\ N_1\to N_2$ bimodule maps such that
$\alpha_1=\alpha_2\circ(\mu\ot_{B'}\nu)$ and $\beta_1=\beta_2\circ
(\nu\ot_B \mu)$.

In a similar way (see \cite{CCMT}), we introduce the categories
$\Morita^{\sq_H}(B,B')$ and $\Morita^H(A,A')$.

We recall the following result, see \cite[Theorems 5.7 and 5.9]{CCMT}.

\begin{theorem} \thlabel{1.6}
Assume that $A$ and $A'$ are right faithfully flat $H$-Galois
extensions of $B$ and $B'$.
\begin{enumerate}
\item The categories $\Morita^H(A,A')$ and $\Morita^{\sq_H}(B,B')$
are equivalent. The equivalence functors send strict contexts to strict contexts.
\item Let $(B,B',M_1,N_1,\alpha_1,\beta_1)$ be strict Morita context.
If $M_1$ has a left $A\sq_H {A'}^{\rm op}$-module structure, then there is a unique
left $A'\sq_H A^{\rm op}$-module structure on $N_1$ such that
$(B,B',M_1,N_1,\alpha_1,\beta_1)$ is a strict
$\sq_H$-Morita context. The corresponding strict $H$-Morita context\\
$(A,B,M,N,\alpha,\beta)$ is given by the following data
$$M= (A\ot {A'}^{\rm op})\ot_{A\square {A'}^{\rm op}}M_1\in {}_A\Mm_{A'}^H;$$
$$N= (A'\ot A^{\rm op})\ot_{A'\square A^{\rm op}}N_1\in
{}_{A'}\Mm_A^H;$$
$$\alpha= (A\ot A^{\rm op})\ot_{A\square A^{\rm op}}\beta_1~~;~~
\beta= (A'\ot {A'}^{\rm op})\ot_{{A'}\square {A'}^{\rm op}}\beta_1.$$
\end{enumerate}
\end{theorem}

\subsection*{Hopf subalgebras} Now let $K$ be a Hopf subalgebra
of $H$. We assume that the antipode of $K$ is bijective, and that
$H$ is faithfully flat as a left $K$-module. Let
$K^+=\Ker(\varepsilon_K)$. It is well-known, and easy to prove
(see \cite[Sec. 1]{Ulbrich}) that
$$\ol{H}=H/HK^+\cong H\ot_K k$$
is a left $H$-module coalgebra, with operations
$$h\cdot \ol{l}=\ol{hl},~~\Delta_{\ol{H}}(\ol{h})=\ol{h}_{(1)}\ot \ol{h}_{(2)},~~
\varepsilon_{\ol{H}}(\ol{h})=\varepsilon(h).$$
The class in $\ol{H}$ represented by $h\in H$ is denoted by $\ol{h}$.
$\ol{1}$ is a grouplike element of $\ol{H}$, and we consider coinvariants with
respect to this element. A right $H$-comodule $M$ is also a right $\ol{H}$-comodule,
by corestriction of coscalars:
$$\rho_{\ol{H}}(m)=m_{[0]}\ot \ol{m}_{[1]}.$$
The $\ol{H}$-coinvariants of $M\in \Mm^H$ are then
\begin{eqnarray*}
M^{{\rm co}\ol{H}}&=&\{m\in M~|~m_{[0]}\ot \ol{m}_{[1]}=m\ot \ol{1}\}\\
&=& \{m\in M~|~\rho(m)\in M\ot K\}\cong M\square_H K.
\end{eqnarray*}
If $A$ is a right $H$-comodule algebra, then $A^{{\rm co}\ol{H}}$
is a right $K$-comodule algebra, and $(A^{{\rm co}\ol{H}})^{{\rm co}K}
=A^{{\rm co}H}$. In \cite[Cor. 7.3]{CCMT}, we have seen the following result,
based on \cite[Remark 1.8]{Schneider1}.

\begin{proposition}\prlabel{1.3}
Let $H$, $K$ and $A$ be as above, and assume that $A$ is a faithfully
flat $H$-Galois extension of $B$. Then $A^{{\rm co}\ol{H}}$ is a faithfully flat
$K$-Galois extension of $B$.
\end{proposition}

Let $i:\ A^{{\rm co}\ol{H}}\to A$ and $j:\ K\to H$ be the inclusion maps. Then
we have a commutative diagram
$$\xymatrix{
A^{{\rm co}\ol{H}}\ot_B A^{{\rm co}\ol{H}}\ar[rr]^{\can_{A^{{\rm co}\ol{H}}}}
\ar[d]_{i\ot_B i}&&
A^{{\rm co}\ol{H}}\ot K\ar[d]^{i\ot j}\\
A\ot_BA\ar[rr]^{\can_A}&&A\ot H}$$
The map $i\ot j$ is injective (here we use the fact that we work over a field $k$).
From the fact that
$\can_{A^{{\rm co}\ol{H}}}$ is an isomorphism, it follows that
$i\ot_B i$ is also injective. For $k\in K$, we then have
$$(\can_A\circ (i\ot_B i))(\gamma_{A^{{\rm co}\ol{H}}})=
((i\ot j)\circ \can_{A^{{\rm co}\ol{H}}})(\gamma_{A^{{\rm co}\ol{H}}})=1\ot j(k),$$
hence
\begin{equation}\eqlabel{1.3.1}
(i\ot_B i)\gamma_{A^{{\rm co}\ol{H}}}(k)=\gamma_A(j(k)).
\end{equation}

\section{Cotensor product of Hopf-Galois extensions}\selabel{2}
Troughout this Section, we assume that $H$ is cocommutative. $\Delta:\ H\to H\ot H$
is a Hopf algebra map, so we can consider $H$ as a Hopf subalgebra of $H\ot H$.
Then $H\ot H$ is a left $H$-module by restriction of scalars.

\begin{lemma}\lelabel{2.1}
$H\ot H$ is faithfully flat as a left $H$-module.
\end{lemma}

\begin{proof}
Let $H\ot \lan H\ran$ be the vector space $H\ot H$, but with left $H$-action
$h(k\ot l)=hk\ot l$. Then $H\ot H$ and $H\ot \lan H\ran$ are isomorphic as left
$H$-modules, and we have the following natural isomorphisms of functors:
$$-\ot_H(H\ot H)\cong -\ot_H(H\ot \lan H\ran)\cong -\ot_k H,$$
an the result follows from the fact that $H$ is faithfully flat
as a $k$-vector space.
\end{proof}

In a similar way, we have an isomorphism $(H\ot H)\ot_H M\cong H\ot M$, for every left
$H$-module $M$. In particular, $k$ is a left $H$-module via the counit $\varepsilon$,
so we have an isomorphism
$$f:\ (H\ot H)\ot_H k\to H,~~f(\ol{h\ot k})=hS(k)$$
of $H$-module coalgebras, with left $H$-action on $H$ given by
$h\cdot k=\varepsilon(h)k$.

\begin{lemma}\lelabel{2.2}
Let $A$ and $A'$ be faithfully flat $H$-Galois extensions of $B$
and $B'$. Then the following statements hold.
\begin{enumerate}
\item $A\ot A'$ is a faithfully flat $H\ot H$-Galois extension of
$B\ot B'$. \item $(A\ot A')^{{\rm co}\ol{H\ot H}}\cong A\sq_H A'.$
\item $(A\sq_H A')^{{\rm co}H}=B\ot B'.$
\end{enumerate}
\end{lemma}

\begin{proof}
(1) We first show that $(A\ot A')^{{\rm co}(H\ot H)}=B\ot B'$. We
have a map
$$f:\ B\ot B'\to (A\ot A')^{{\rm co}(H\ot H)},~~f(b\ot b')=b\ot b'.$$
$B\ot B'=(A\ot B')\cap(B\ot A')$ and $(A\ot A')^{{\rm co}(H\ot H)}$ are both subspaces
of $A\ot A'$, so it suffices to show that $f$ is surjective. Take
$\sum_i a_i\ot a'_i\in (A\ot A')^{{\rm co}(H\ot H)}$. Then
$$\sum_i a_{i[0]}\ot a'_{i[0]}\ot a_{i[1]}\ot a'_{i[1]}= \sum_i a_i\ot a'_i\ot 1\ot 1.$$
Applying $\varepsilon$ to the fourth tensor factor, we find
$$\sum_i a_{i[0]}\ot a'_{i}\ot a_{i[1]}= \sum_i a_i\ot a'_i\ot 1.$$
This means that $\sum_i a_i\ot a'_i\in B\ot A'$. In a similar way, we find that
$\sum_i a_i\ot a'_i\in A\ot B'$.\\
It is easy to show that $\can_{A\ot A'}$ is bijective. Finally $A\ot A'$ is faithfully
flat as a right $B\ot B'$-module: $B\ot A'$ is faithfully flat as a right $B\ot B'$-module
because for every left $B\ot B'$-module $M$ there is a natural isomorphism
$(B\ot A')\ot_{B\ot B'} M\cong A'\ot_{B'} M$. Similarly, 
$A\ot A'$ is faithfully flat as a right $B\ot A'$-module. Then apply the following
general property: if $f:\ A\to B$ and $g:\ B\to C$ are algebra morphisms, and $B/A$
and $C/B$ are faithfully flat, then $C/A$ is faithfully flat.

(2) We can apply \prref{1.3}, with $H$ replaced by $H\ot H$, $K$
by $H$ and $A$ by $A\ot A'$. Note that $\sum_i a_i\ot a'_i\in
(A\ot A')^{{\rm co}\ol{H\ot H}}$ if and only if
$$\sum_i a_{i[0]}\ot a'_{i[0]}\ot a_{i[1]}S(a'_{i[1]})= \sum_i a_i\ot a'_i\ot 1,$$
or
$$\sum_i a_{i[0]}\ot a'_i\ot a_{i[1]}= \sum_i a_i\ot a'_{i[0]}\ot a'_{i[1]},$$
which means precisely that $\sum_i a_i\ot a'_i\in A\sq_H A'$.

(3) We know that $A\sq_H A'$ is a right $H$-comodule algebra with
structure map $\rho$ given by
\begin{equation}\eqlabel{2.4.1}
\rho(\sum_i a_i\ot a'_i)= \sum_i a_{i[0]}\ot a_i'\ot a_{i[1]}=
\sum_i a_i\ot a'_{i[0]}\ot a'_{i[1]}.
\end{equation}
Take $x=\sum_i a_i\ot a'_i\in (A\sq_H A')^{{\rm co}H}$. It follows
from \equref{2.4.1} that $x\in (B\ot A')\cap (A\ot B')=B\ot B'$.
\end{proof}

Combining these observations with \prref{1.3}, we obtain the
following result, which is well-known in the situation where
$B=B'=k$.

\begin{theorem}\thlabel{2.5}
Let $A$ and $A'$ be faithfully flat $H$-Galois extensions of $B$ and $B'$. Then
$A\sq_H A'$ is a faithfully flat $H$-Galois extension of $B\ot B'$.
\end{theorem}

We want to apply this theorem in the case when $A'$ is the
opposite algebra $A^{{\rm op}}$. Since $H$ is cocommutative,
$A^{{\rm op}}$ is a right $H$-comodule algebra, with coaction
$\rho$ given by
$$\rho(a)= a_{[0]}\ot S(a_{[1]}).$$

\begin{lemma}\lelabel{2.6}
If $A$ is a faithfully flat $H$-Galois extension of $B$, then $A^{\rm op}$ is
a faithfully flat $H$-Galois extension of $B^{\rm op}$.
\end{lemma}

\begin{proof} The map
$\can_{A^{\rm op}}:\ A^{\rm op}\ot_{B^{\rm op}} A^{\rm op}\to A^{\rm op}\ot H$ is given by
$$\can_{A^{\rm op}}(a\ot a')=a'_{[0]}a\ot S(a'_{[1]})=(A^{\rm op}\ot S)\circ \can'_A.$$
Then $\can_{A^{\rm op}}$ is bijective since $\can'_A$ and $S$ are
bijective. We know from \thref{1.1} that $A\in \Mm_B$ is
faithfully flat, and this implies that
 $A^{\rm op}\in {}_{B^{\rm op}}\Mm$ is faithfully flat. It then follows from \thref{1.1}
 that $A^{\rm op}$ is also
a faithfully flat $H$-Galois extension.
\end{proof}

\begin{proposition}\prlabel{2.7}
Let $A$ be a faithfully flat $H$-Galois extension of $B$. Then
$A^{\sq e}:=A\sq_H A^{\rm op}$ is a faithfully flat $H$-Galois
extension of the enveloping algebra $B^e:=B\ot B^{\rm op}$.
Moreover, the element
\begin{equation}\eqlabel{2.7.1}
\gamma_{A^{\sq e}}(h):= \sum_{i,j} (l_i(h_{(1)})\ot
r_j(h_{(2)}))\ot_{B\ot B^{\rm op}} (r_i(h_{(1)})\ot l_j(h_{(2)}))
\end{equation}
belongs to $ A^{\sq e}\ot_{B^e} A^{\sq e}$.
\end{proposition}

\begin{proof}
First observe that
$\can_{A^e}:\ A^e\ot_{B^e}A^e\to
A^e\ot H\ot H$ is given by
$$\can_{A\ot A^{\rm op}}((a\ot b)\ot (a'\ot b'))=
aa'_{[0]}\ot b'_{[0]}b\ot a'_{[1]}\ot S(b'_{[1]}).$$ Recall the
notation $\gamma_A(h):=\sum_i l_i(h)\ot_B r_i(h).$ Then we compute
that
\begin{eqnarray*}
&&\hspace*{-10mm}
\can_{A^e}\Bigl(\sum_{i,j} (l_i(h_{(1)})\ot r_j(h_{(2)}))\ot_{B^e} (r_i(h_{(1)})\ot l_j(h_{(2)}))
\Bigr)\\
&=&
\sum_{i,j} l_i(h_{(1)})r_i(h_{(1)})_{[0]}\ot l_j(h_{(2)})_{[0]}r_j(h_{(2)})\ot r_i(h_{(1)})_{[1]}
\ot S(l_j(h_{(2)})_{[1]})\\
&\equal{(\ref{eq:1.2.1},\ref{eq:1.2.4})}&
\sum_j 1\ot l_j(h_{(3)})r_j(h_{(3)})\ot h_{(1)}\ot S(S(h_{(2)}))\equal{\equref{1.2.5}} 1\ot 1\ot \Delta(h).
\end{eqnarray*}
Let $i:\ A^{\sq e}\to A^e$ be the canonical injection. It follows from \equref{1.3.1}
that
\begin{eqnarray*}
&&\hspace*{-2cm}
(i\ot_{B\ot B^{\rm op}}i)(\gamma_{A^{\sq e}}(h))=\gamma_{A^e}(\Delta(h))\\
&=& \sum_{i,j} (l_i(h_{(1)})\ot r_j(h_{(2)}))\ot_{B^e}
(r_i(h_{(1)})\ot l_j(h_{(2)})),
\end{eqnarray*} and the statement is proved.
\end{proof}

\section{Cleft extensions and the lifting Theorem}\selabel{4}
In this Section, we adapt and review the results from \cite{MiSt}, going back to older results
from graded Clifford theory, see \cite{Dade}.

\subsection*{Cleft extensions}
\begin{proposition}\prlabel{4.1}
Let $H$ be a Hopf algebra, $A$ a right $H$-comodule algebra, and $B=A^{{\rm co}H}$.
We have a category $\Cc_A$, with two objects ${\bf 1}$ and ${\bf 2}$,
and morphisms
$$\Cc_A({\bf 1},{\bf 1})=\Hom(H,B)~~;~~\Cc_A({\bf 1},{\bf 2})=\Hom^H(H,A);$$
$$\Cc_A({\bf 2},{\bf 1})=\{u: H\to A~|~\rho(u(h))=u(h_{(2)})\ot S(h_{(1)}),~{\rm for~all~}h\in
H\};$$
$$\Cc_A({\bf 2},{\bf 2})=\{w: H\to A ~|~
 \rho(w(h))=w(h_{(2)})\ot S(h_{(1)})h_{(3)},~{\rm for~all~}h\in
H\}.$$
The composition of morphisms is given by the convolution product.
\end{proposition}

Recall that $A$ is called $H$-cleft if there exists a convolution invertible
$t\in \Hom^H(H,A)$, or, equivalently, if ${\bf 1}$ and ${\bf 2}$ are isomorphic
in $\Cc_A$. Then $t(1)^{-1}=u(1)$, and
$t'=u(1)t\in \Hom^H(H,A)$ has convolution inverse $ut(1)$, and $t'(1)=1$.
So if $A$ is $H$-cleft, then there exists a convolution invertible $t\in \Hom^H(H,A)$
with $t(1)=1$.\\
If $H$ is cocommutative, then $\Cc_A({\bf 1},{\bf 1})=\Cc_A({\bf 2},{\bf 2})$.\\
If $t\in \Hom^H(H,A)$ is an algebra map, then $t$ is convolution invertible
(with convolution inverse $t\circ S$), so $A$ is $H$-cleft. Consider the space
\[\Omega_A=\{t\in \Hom^H(H,A)~|~t~{\rm
is~an~algebra~map}\}.\]
We have the following equivalence
relation on $\Omega_A$: $t_1\sim t_2$ if and only if there exists $b\in U(B)$
such that $bt_1(h)=t_2(h)b$, for all $h\in H$. We denote
$\ol{\Omega}_A=\Omega_A/\sim$.\\
Take $t\in \Hom^H(H,A)$ with convolution inverse $u$ such that $t(1_H)=1_A$,
and consider the map
$$\omega_t:\ H\ot B\to B, \quad\omega_t(h\ot b)= t(h_{(1)})b u(h_{(2)}).$$
Assume that $\Omega_A\neq \emptyset$, and fix $t_0\in \Omega_A$
with convolution inverse $u_0$. Now
consider the bijection
$$F:\ \Cc_A({\bf 1},{\bf 1})=\Hom(H,B)\to \Cc_A({\bf 1},{\bf 2})=\Hom^H(H,A),$$
$F(v)=v*t_0$, $F^{-1}(t)=t*u_0$. It is then easy to show that $F(v)\in \Omega_A$
if and only if
\begin{equation}\eqlabel{4.2.0}
v(hk)=v(h_{(1)})\omega_{t_0}(h_{(2)}\ot v(k))
\end{equation}
and $v(1_H)=1_B$. If \equref{4.2.0} holds, then $v(1_H)=1_B$ if and only if
$v$ is convolution invertible. Moreover, $F(v)\sim t_0$ if and only if
$v(h)=\omega_{t_0}(h\ot b)b^{-1}$ for some invertible $b\in B$.\\
We will now discuss when $F^{-1}(\Omega_A)$ is a subgroup of $\Hom(H,B)$.

\begin{proposition}\prlabel{4.2}
Let $H$ be cocommutative, and let $A$ be an $H$-cleft right $H$-comodule algebra.
Assume that $B=A^{{\rm co}H}$ is commutative.
Choose $t\in \Hom^H(H,A)$ with convolution inverse $u$, such that $t(1)=1$ and,
a fortiori, $u(1_H)=1_A$.
Then we have the following properties.
\begin{enumerate}
\item $\omega_t$ is independent of the choice of $t$;
\item $ab=\omega_t(a_{[1]}\ot b)a_{[0]}$, for all $a\in A$ and $b\in
B$.
\end{enumerate}
\end{proposition}

If $\Omega_A\neq \emptyset$, then we have an algebra map $t\in \Hom^H(H,A)$,
and then the map $\omega_t$ defines a left $H$-module algebra
structure on $B$, and we can consider the Sweedler
cohomology groups $H^n(H,B)$, see \cite{Sweedler}. We then denote
$h\cdot b=\omega_t(h\ot b)$.

\begin{proposition}\prlabel{4.3}
Assume that $\Omega_A\neq \emptyset$. Then $\Omega_A\cong Z^1(H,B)$ and
$\ol{\Omega}_A\cong H^1(H,B)$.
\end{proposition}

\begin{proof} (sketch)
If $H$ is cocommutative and $B$ is commutative, then \equref{4.2.0} is equivalent to
$$v(hk)=(h_{(1)}\cdot v(k))v(h_{(2)}),$$
which is precisely the condition that $v$ is a Sweedler 1-cocycle.
\end{proof}

\begin{proposition}\prlabel{4.3b}
Now assume that $B=k$; it is not necessary that $H$ is cocommutative. If $\Omega_A\neq
\emptyset$, then $\Omega_A\cong {\rm Alg}(H,k)$.
\end{proposition}

\begin{proof}
In this situation, $\omega_t(h\ot b)=\varepsilon(h)b$, for every choice of $t$. Then
\equref{4.2.0} is equivalent to $v(hk)=v(h)v(k)$, and the result follows.
\end{proof}

Suppose that $A$ is $H$-cleft. Pick a convolution invertible $t\in \Hom^H(H,A)$
such that $t(1)=1$. Then consider
$$\sigma:\ H\ot H\to B,~~\sigma(h\ot k)=t(h_{(1)})t(k_{(1)})u(h_{(2)}k_{(2)}).$$
Let $B\#_\sigma H$ be equal to $B\ot H$ as a vector space, with right $H$-coaction
$\rho=B\ot \Delta$, and with multiplication
$$(b\#h)(c\# k)= b(h\cdot c)\sigma(h_{(1)}\ot k_{(1)})h_{(2)}k_{(2)}.$$

\begin{proposition}\prlabel{4.4}
The map $\phi:\ B\#_\sigma H\to A$, $\phi(b\#h)=bt(h)$ is an isomorphism of
right $H$-comodule algebras. The inverse of $\phi$ is given by the formula
$\phi^{-1}(a)=a_{[0]}u(a_{[1]})\# a_{[2]}$.
Let $\sigma\in Z^2(H,B)$. The following statements are equivalent:
\begin{enumerate}
\item $\sigma\in B^2(H,B)$;
\item there exists an algebra map $t'\in \Hom^H(H,A)$;
\item $A\cong B\#_{\varepsilon\ot\varepsilon}H$.
\end{enumerate}
\end{proposition}

\subsection*{The Militaru-Stefan lifting Theorem}
Let $A$ be a faithfully flat $H$-Galois
extension of $B=A^{{\rm co}H}$. ${}_A\Mm^H$ will denote the
category of (left-right) relative Hopf modules. Let $P,Q\in
{}_A\Mm^H$. A left $A$-linear map $f:\ P\to Q$ is called {\it
rational} if there exists a (unique) element $f_{[0]}\ot
f_{[1]}\in {}_A\Hom(P,Q)\ot H$ such that
\begin{equation*}
f_{[0]}(p)\ot f_{[1]}=f(p_{[0]})_{[0]}\ot S^{-1}(p_{[1]})f(p_{[0]})_{[1]},
\end{equation*}
or, equivalently,
\begin{equation}\eqlabel{4.1.2}
\rho(f(p))=f_{[0]}(p_{[0]})\ot p_{[1]} f_{[1]},
\end{equation}
for all $p\in P$. The subset of ${}_A\Hom(P,Q)$ consisting of rational maps is denoted by
 ${}_A\HOM(P,Q)$. This is a right $H$-comodule, and ${}_A\END(P)^{\rm op}$ is a right
 $H$-comodule algebra.\\
 Now take $M\in {}_B\Mm$. Then $A\ot_B M\in {}_A\Mm^H$, and $E={}_A\END(A\ot_B M)^{\rm op}$ is a right
 $H$-comodule algebra. From the category equivalence between $ {}_B\Mm$ and ${}_A\Mm^H$,
 it follows that \[F:=E^{{\rm co}H}={}_A\End^H(A\ot_B M)^{\rm op}\cong {}_B\End(M)^{\rm
 op}.\]
 $B$ can be viewed as a right $H$-comodule algebra, with trivial coaction $\rho(b)=
 b\ot 1$, for all $b\in B$, so we can consider the category of relative Hopf modules
 ${}_B\Mm^H$. If $M$ is a left $B$-module, then $A\ot_BM$ and $M\ot H$ are objects
 of  ${}_B\Mm^H$. $\Dd_M$ will be the full subcategory of ${}_B\Mm^H$, with two
 objects $A\ot_BM$ and $M\ot H$. We then have the following result.

 \begin{theorem}\thlabel{4.5}
 Let $A$ be a faithfully flat $H$-Galois
extension of $B=A^{{\rm co}H}$ and $M\in {}_B\Mm$. Then the categories $\Cc_E$ and
$\Dd_M$ are anti-isomorphic.
\end{theorem}

\begin{proof} (sketch)
We define a contravariant functor $\alpha:\ \Cc_E\to \Dd_M$ at the objects level in the
following obvious way: $\alpha({\bf 1})= M\ot H$ and $\alpha({\bf 2})=A\ot_BM$. Before
we state the definition at the morphisms level, we observe that we have
two natural isomorphisms
$$\beta_1:\  {}_B\Hom(A\ot_B M,M)\to {}_B\Hom^H(A\ot_B M, M\ot H);$$
$$\beta_2:\ {}_B\Hom(M\ot H,M)\to {}_B\End^H(M\ot H)$$
defined as follows:
$$\beta_1(\phi)(a\ot_B m)=\phi(a_{[0]}\ot_B m)\ot a_{[1]}~~;~~
\beta_1^{-1}(\varphi)=(M\ot\varepsilon)\circ\varphi;$$
$$\beta_2(\Theta)(m\ot h)=
\Theta(m\ot h_{(1)})\ot h_{(2)}~~;~~
\beta_2^{-1}(\theta)=(M\ot\varepsilon)\circ\theta.$$
Consider $\eta_M:\ M\to (A\ot_BM)^{{\rm co}H}$, the unit of the adjunction $(F_2,G_2)$
(see \seref{1}) evaluated at $M$. Since $F_2$ is an equivalence of categories, $\eta_M$
is an isomorphism. We have an isomorphism
$$\tilde{\alpha}_{\bf 11}:\ \Cc_E({\bf 1},{\bf 1})=\Hom(H,E^{{\rm co}H})\to {}_B\Hom(M\ot H,M),$$
given by the formulas
\begin{eqnarray*}
&&\tilde{\alpha}_{\bf 11}(v)(m\ot h)=\eta_M^{-1}(v(h)(1\ot_B m));\\
&&\tilde{\alpha}_{\bf 11}^{-1}(\Theta)(h)(a\ot_B m)=a\ot_B \Theta(m\ot h).
\end{eqnarray*}
We then define $\alpha_{\bf 11}=\beta_2\circ \tilde{\alpha}_{\bf 11}$. The isomorphism
 \[\alpha_{\bf 12}:\ \Cc_E({\bf 1},{\bf 1})=\Hom^H(H,E)\to {}_B\Hom^H(M\ot
H,A\ot_B M)\]
is given by the formulas
$$\alpha_{\bf 12}(t)(m\ot h)=t(h)(1\ot_Bm)~~;~~(\alpha_{\bf 12}^{-1}(\psi)(h))(a\ot_B m)=
a\psi(m\ot h).$$
We have an isomorphism
$$\tilde{\alpha}_{\bf 21}:\ \Cc_E({\bf 2},{\bf 1})\to {}_B\Hom(A\ot_BM,M),$$
given by the formulas
\begin{eqnarray*}
&&\tilde{\alpha}_{\bf 21}(u)(a\ot_B m)=\eta_M^{-1}(u(a_{[1]})(a_{[0]}\ot_Bm));\\
&& \bigl(\tilde{\alpha}_{\bf 21}^{-1}(\phi)(h)\bigr)(a\ot_B m)= \sum_i
al_i(h)\ot_B \phi(r_i(h)\ot_B m).
\end{eqnarray*}
We then define $\alpha_{\bf 21}=\beta_1\circ \tilde{\alpha}_{\bf 21}$.
Finally, the isomorphism
$$\alpha_{\bf 22}:\ \Cc_E({\bf 2},{\bf 1})\to {}_B\End^H(A\ot_BM)^{\rm op}$$
is given by the formulas
\begin{eqnarray*}
&&\alpha_{\bf 22}(w)(a\ot_B m)=w(a_{[1]})(a_{[0]}\ot_B m);\\
&&(\alpha_{\bf 22}^{-1}(\kappa))(h)(a\ot_B m)=
\sum_i al_i(h)\kappa(r_i(h)\ot_B m).
\end{eqnarray*}
A long computation shows that $\alpha_{\bf 22}$ is a well-defined isomorphism,
and that $\alpha$ is a functor.
\end{proof}

Recall from \cite{Schneider1} that $M\in {}_B\Mm$ is called $H$-stable if $A\ot_B M$ and $M\ot H$ are
isomorphic as left $B$-modules and right $H$-comodules, or, equivalently,
the two objects of $\Dd_M$ are isomorphic.
From \thref{4.5}
we immediately deduce the following result.

\begin{corollary}\colabel{4.10}
$M\in {}_B\Mm$ is $H$-stable if and only if there exists a convolution invertible
$t\in \Hom^H(H,E)$.
\end{corollary}

 Assume that $M\in {}_B\Mm$ is $H$-stable. Then there is an isomorphism
$\varphi:\ A\ot_B M\to M\ot H$ in ${}_B\Mm^H$. Let $\psi=\varphi^{-1}$,
$\phi=(M\ot\varepsilon)\circ \varphi$, $t=\alpha_{\bf 12}^{-1}(\psi)$, $u=\alpha_{\bf 21}^{-1}(\varphi)$.
Then the following assertions are equivalent.
\begin{enumerate}
\item $t(1)=1$;
\item $u(1)=1$;
\item $\psi(m\ot 1)=1\ot_B m$, for all $m\in M$;
\item $\phi(1\ot_B m)=m$, for all $m\in M$.
\end{enumerate}
Indeed, the equivalences
$1) \Longleftrightarrow 2)$ and $3) \Longleftrightarrow 4)$ are obvious, and
$1) \Longleftrightarrow 3)$ follows immediately from the definition of $\alpha_{\bf 12}$ and
$\alpha_{\bf 21}^{-1}$.\\
We have seen (cf. comments following \prref{4.1}) that $t'$ and $u'$ given by
$t'(h)=t(h)\circ u(1)$ and $u'(h)=t(1)\circ u(h)$
are convolution inverses, satisfying the additional condition $t'(1)=u'(1)=1$. Thus
$\psi'=\alpha_1(t')$ satisfies (3), and $\phi'=\tilde{\alpha}_2(u')$ satisfies (4).
$\psi'$ and $\phi'$ can be computed explicitly, using the formulas given in the proof of
\thref{4.5}:
$$\psi'(m\ot h)=\psi(\phi(1\ot_B m)\ot h)~~;~~
1\ot_B \phi'(a\ot_B m)=\psi(\phi(a\ot_B m)\ot 1).$$ $\psi'$ and $\varphi'$ are composition
inverses. The proof of the following result is now a straightforward exercise.

\begin{proposition}\prlabel{4.11}
Take $\phi\in {}_B\Hom(A\ot_BM,M)$, and let
$u=\tilde{\alpha}_{\bf 21}^{-1}(\phi) \in \Cc({\bf 2},{\bf 1})$ and $t=u\circ
S^{-1}\in \Cc({\bf 1},{\bf 2})=\Hom^H(H,E)$. Then the following statements are
equivalent:
\begin{enumerate}
\item $\phi:\ A\ot_B M\to M$, $\phi(a\ot_B m)=a\cdot m$ is an associative left $A$-action on $M$;
\item $u$ is an anti-algebra map;
\item $t$ is an algebra map.
\end{enumerate}
\end{proposition}

\begin{proposition}\prlabel{4.12}
For $i=1,2$, take $\phi_i\in {}_B\Hom(A\ot_B M,M)$, and consider
$u_i=\tilde{\alpha}_{\bf 21}^{-1}(\phi_i)
\in \Hom^S(H,E)$ and $t_i=u_i\circ S^{-1}\in \Hom^H(H,E)$.
Let $M_i=M$ as a left $B$-module, with left $A$-action defined by $\phi_i$.
Then $M_1\cong M_2$ if and only if $t_1\sim t_2$.
\end{proposition}

\begin{proof} We have that
$t_1\sim t_2$ if and only if there exists an invertible map $f\in
{}_B\End(M)\cong E^{{\rm co}H}$ such that $t_1(h)\circ (A\ot_B
f)= (A\ot_B f)\circ t_2(h)$, or, equivalently, $u_1(h)\circ
(A\ot_B f)= (A\ot_B f)\circ u_2(h)$, for all $h\in H$. This
implies that
\begin{eqnarray*}
&&\hspace*{-2cm}
1\ot_B \phi_1(a\ot_B f(m))=u_1(a_{[1]})(a_{[0]}\ot_B f(m))\\
&=& (u_1(a_{[1]})\circ (A\ot_B f))(a_{[0]}\ot_B m)\\
&=&((A\ot_B f)\circ u_2(a_{[1]}))(a_{[0]}\ot_B m)\\
&=& 1\ot_B f(\phi_2(a\ot_B f(m))),
\end{eqnarray*}
and $\phi_1(a\ot_B f(m))=f(\phi_2(a\ot_B f(m)))$, for all $a\in A$ and $m\in M$,
which means that $f:\ M_2\to M_1$ is an isomorphism of left $A$-modules.\\
Conversely, let $f:\ M_2\to M_1$ is an isomorphism of left $A$-modules.
Then $f:\ M\to M$ is left $B$-linear, so $f\in {}_B\End(M)$. Then we have,
for all $h\in H$, $a\in A$ and $m\in M$, that
\begin{eqnarray*}
&&\hspace*{-1cm}
u_1(h)(a\ot_B f(m))=\sum_i al_i(h)\ot_B \phi_1(r_i(h)\ot_B f(m))\\
&=&\sum_i al_i(h)\ot_B f(\phi_2(r_i(h)\ot_B m))
= (A\ot_B f)(u_2(h)(a\ot_B m)),
\end{eqnarray*}
hence $u_1(h)\circ (A\ot_B f)=(A\ot_B f)\circ u_2(h)$, as needed.
\end{proof}

As an immediate consequence, we obtain the Militaru-\c Stefan lifting Theorem.

\begin{corollary}\colabel{4.13}
Let $A$ be a faithfully flat $H$-Galois extension of $B=A^{{\rm co}H}$
and $M\in {}_B\Mm$.
There is a bijective correspondence between the isomorphism classes of left $A$-module
structures on $M$ extending the $B$-module structure on $M$ and the elements
of $\ol{\Omega}_E$.
\end{corollary}

\begin{example}\exlabel{4.14}
Let $A$ be an $H$-Galois object, that is, $A^{{\rm co}H}=k$, and $M=k$.
Then $E={}_A\End(A)^{\rm op}\cong A$ as an $H$-comodule algebra,
and $E^{{\rm co}H}=A^{{\rm co}H}=k$. The map $\tilde{\alpha}_{\bf 21}:\
\Cc_A({\bf 1},{\bf 2})$ and its inverse are given by the formulas
$$\tilde{\alpha}_{\bf 21}(u)(a)=u(a_{[1]})a_{[0]}\in A^{{\rm co}H}=k;$$
$$\tilde{\alpha}_{\bf 21}^{-1}(\phi)(h)=\sum_i l_i(h)\phi(r_i(h)).$$
$\phi\in A^*$ defines an $A$-action $A\ot k\to k$ if and only if $\phi$ is
an algebra map. It follows from \coref{4.13} that $\ol{\Omega}_A={\rm Alg}(A,k)$.
If $\ol{\Omega}_A\neq \emptyset$, then it follows from \prref{4.3b} that
${\rm Alg}(A,k)\cong {\rm Alg}(H,k)$. The correspondence goes as follows.
Fix $\phi_0\in {\rm Alg}(A,k)$. $\phi\in {\rm Alg}(A,k)$ corresponding to
$v\in {\rm Alg}(H,k)$ is given by the formula
$$\phi(a)=v(S(a_{[1]}))l_i(a_{[2]})\phi_0(r_i(a_{[2]}))a_{[0]}.$$
\end{example}

\section{Picard groups}\selabel{5}
\subsection*{The Picard group of an $H$-comodule algebra}
Consider a  Hopf algebra $H$ with bijective antipode and an $H$-comodule
algebra $A$. Let $\dul{\Pic}^H(A)$ be the category with strict
$H$-Morita contexts of the form $(A,A,P,Q,\alpha,\beta)$ as
objects. A morphism between $(A,A,P_1,Q_1,\alpha_1,\beta_1)$ and
$(A,A,P_2,Q_2,\alpha_2,\beta_2)$ consists of a couple $(f,g)$,
with $f:\ P_1\to P_2$, $g:\ Q_1\to Q_2$ $H$-colinear $A$-bimodule
isomorphisms such that $\alpha_1=\alpha_2\circ (f\ot_B g)$ and
$\beta_1=\beta_2\circ (g\ot_A f)$. Note that
 $\dul{\Pic}^H(A)$ has the structure of monoidal category, where
the tensor product is given by the formula
\begin{eqnarray*}
&&\hspace*{-15mm}
(A,A,P_1,Q_1,\alpha_1,\beta_1)\ot (A,A,P_2,Q_2,\alpha_2,\beta_2)=
\bigl(A,A,P_1\ot_AP_2,\\
&&~~~~~Q_2\ot_AQ_1,\alpha_1\circ (P_1\ot_A \alpha_2\ot_A Q_1),
\beta_2\circ (Q_2\ot_A\beta_1\ot_A P_1)\bigr).
\end{eqnarray*}
The unit object is $(A,A,A,A,A,A)$. Every object
$(A,A,P_1,Q_1,\alpha_1,\beta_1)$ of
$\dul{\Pic}^H(A)$ has an inverse, namely $(A,A,Q_1,P_1,\beta_1,\alpha_1)$.\\
Up to isomorphism, a strict $H$-Morita context is completely determined by one of its
underlying bimodules; therefore, we use the shorter notation
$\ul{P}_1=(A,A,P_1,Q_1,\alpha_1,\beta_1)$.
$\Pic^H(A)=K_0\dul{\Pic}^H(A)$, the set of isomorphism classes in $\dul{\Pic}^H(A)$,
is a group under the operation induced by the tensor product, and is called
the $H$-Picard group of $A$. If $H=k$, and $B$ is a $k$-algebra, then
$\Pic^k(B)=\Pic(B)$ is the classical Picard group of $B$.

\subsection*{The $\sq$-Picard group of $B$.} Let $M,N\in
{}_{A^{\sq e}}\Mm$. In \cite{CCMT}, it is shown that $M\ot_B N\in
{}_{A^{\sq e}}\Mm$. We will need an explicit formula for the
$A^{\sq e}$-action on $M\ot_B N$, given in \prref{5.11} below.

 In
the proof \cite[Theorem 2.4]{CCMT}, it is shown that we have an
isomorphism
$$\alpha_N:\ A\ot_B N\to A^e\ot_{A^{\sq e}} N,\quad
\alpha_N(a\ot_B n)=(a\ot 1)\ot_{A^{\sq e}} n.$$
We claim that the inverse $\alpha_N^{-1}$ of $\alpha_N$ is given by the formula
$$\alpha_N^{-1}((d\ot e)\ot_{A^{\sq e}} n)=
\sum_i dl_i(S(e_{[1]}))\ot_B (r_i(S(e_{[1]}))\ot e_{[0]})\cdot n.$$
It follows from \leref{5.4} that $\alpha_N^{-1}$ is well-defined. Using the property
that $\gamma_A(1_H)=1_A\ot_B 1_A$, we find that
$$(\alpha_N^{-1}\circ \alpha_N)(a\ot_B n)
=\alpha_N^{-1}((a\ot1)\ot_{A^{\sq e}}n)=a\ot_B n.$$
We also compute that
\begin{eqnarray*}
&&\hspace*{-15mm}
(\alpha_N\circ\alpha_N^{-1})((d\ot e)\ot_{A^{\sq e}}n)\\
&=&\alpha_N\Bigl(\sum_i dl_i(S(e_{[1]}))\ot_B (r_i(S(e_{[1]}))\ot e_{[0]})\cdot n\Bigr)\\
&=& \sum_i (dl_i(S(e_{[1]}))\ot 1)\ot_{A^{\sq e}} (r_i(S(e_{[1]}))\ot e_{[0]})\cdot n\\
&=& \bigl(\sum_i dl_i(S(e_{[1]}))r_i(S(e_{[1]}))\ot e_{[0]}\bigr)\ot_{A^{\sq e}} n
\equal{\equref{1.2.5}}(d\ot e)\ot_{A^{\sq e}}n.
\end{eqnarray*}
Using $\alpha_N$, the left $A^e$-action on $A^e\ot_{A^{\sq e}} N$ can be transported
to a left $A^e$-action on $A\ot_B N$:
\begin{eqnarray*}
&&\hspace*{-5mm}
 (d\ot e)(a\ot_B n) = \alpha_N^{-1}\bigl((d\ot e)\alpha_N(a\ot_B
n)\bigr)  \\
&=&\alpha_N^{-1}\bigl((da\ot e)\ot_{A^{\sq e}} n\bigr)
= \sum_i da l_i(S(e_{[1]}))\ot_B (r_i(S(e_{[1]}))\ot
e_{[0]})\cdot n.
\end{eqnarray*}
If $M,N\in {}_{A^{\sq e}}\Mm$, then $A^e\ot _{A^{\sq e}}M, A^e\ot _{A^{\sq e}}N
\in {}_A\Mm_A^H$, hence
$$(A^e\ot _{A^{\sq e}}M)\ot_A(A^e\ot _{A^{\sq e}}N)\cong (A\ot_B M)\ot_A (A\ot_B N)
\cong A\ot_BM\ot_B N$$ in the category ${}_A\Mm_A^H.$ On $(A\ot_B
M)\ot_A (A\ot_B N)$, the $A$-bimodule structure (or left
$A^e$-module structure) is given by the formula
\begin{eqnarray*}
&&\hspace*{-20mm}
(d\ot e)\cdot\bigl((a\ot_B m)\ot_A(a'\ot_Bn)\bigr)\\
&=& (d\ot 1)\cdot (a\ot_B m)\ot_A (a\ot e)\cdot (a'\ot_B n)\\
&=& \sum_i (da\ot_B m)\ot_A \Bigl(a'l_i(S(e_{[1]}))\ot_B \bigl(r_i(S(e_{[1]}))\ot e_{[0]}\bigr)\cdot n
\Bigr).
\end{eqnarray*}
We transport this left $A^e$-module structure to $A\ot_BM\ot_B N$:
\begin{eqnarray*}
&&\hspace*{-20mm}
(d\ot e)\cdot (a\ot_B m\ot_B n)\\
&=& \sum_i (1\ot l_i(S(e_{[1]})))\cdot (da\ot_B m)\ot_B (r_i(S(e_{[1]}))\ot e_{[0]})\cdot n\\
&=& \sum_{i,j} dal_j\Bigl(S\bigl(l_i(S(e_{[1]}))_{[1]}\bigr)\Bigr)\\
&&\hspace*{2cm}\ot_B
\Bigr(r_j\Bigl(S\bigl(l_i(S(e_{[1]}))_{[1]}\bigr)\Bigr)\ot r_i(S(e_{[1]}))_{[0]}\Bigr)\cdot m\\
&&\hspace*{2cm}
\ot_B \Bigl(r_i(S(e_{[1]}))\ot e_{[0]}\Bigr)\cdot n\\
&\equal{\equref{1.2.4}}&
\sum_{i,j} da l_j(S(e_{[1]}))\ot_B \bigl(r_j(S(e_{[1]}))\ot l_i(S(e_{[2]}))\bigr)\cdot m\\
&&\hspace*{2cm}
\ot_B (r_i(S(e_{[2]})\ot e_{[0]})\cdot n.
\end{eqnarray*}
Now take $\sum_k a_k\ot a'_k\in A^{\sq e}$. Using the above formula, we compute
that
\begin{eqnarray*}
&&\hspace*{-10mm}
(\sum_k a_k\ot a'_k)\cdot (1\ot_B m\ot_B n)
= \sum_{i,j,k} a_kl_j(S(a'_{k[1]}))\\
&&
\ot_B \bigl(r_j(S(a'_{k[1]}))\ot l_i(S(a'_{k[2]}))\bigr)\cdot m
\ot_B \bigl(r_i(S(a'_{k[2]}))\ot a'_{k[0]})\cdot n
\end{eqnarray*}
\begin{eqnarray*}
&=& \sum_{i,j,k} a_{k[0]}l_j(S(a_{k[1]}))
\ot_B \bigl(r_j(S(a_{k[1]}))\ot l_i(S(a_{k[2]}))\bigr)\cdot m\\
&&\hspace*{2cm}
\ot_B \bigl(r_i(S(a_{k[2]}))\ot a'_{k})\cdot n\\
&\equal{\equref{1.2.6}}&
1\ot_B (a_{k[0]}\ot l_i(a_{k[1]}))\cdot m \ot_B (r_i(a_{k[1]})\ot a'_k)\cdot n
\end{eqnarray*}
The map \[M\ot_BN\to (A\ot_B M\ot_BN)^{{\rm co}H}, \quad m\ot_B
n\mapsto 1\ot_Bm\ot_B n\] is an isomorphism. Hence the left
$A^{\sq e}$-action on $A\ot_B M\ot_BN$ restricts to an action on
$(A\ot_B M\ot_BN)^{{\rm co}H}$, and defines an action on
$M\ot_BN$. We can summarize this as follows.

\begin{proposition}\prlabel{5.11}
Let $M,N\in {}_{A^{\sq e}}\Mm$. Then we have the following action on $M\ot_B N$:
\begin{equation}\eqlabel{5.11.1}
(\sum_k a_k\ot a'_k)\cdot (m\ot_B n)=
(a_{k[0]}\ot l_i(a_{k[1]}))\cdot m \ot_B (r_i(a_{k[1]})\ot a'_k)\cdot n.
\end{equation}
\end{proposition}

Now let $\dul{\Pic}^{\sq_H}(B)$ be the category with strict
$\sq_H$-Morita contexts of the form $(B,B,M,N,\gamma,\delta)$ as
objects. A morphism between the $\sq_H$-Morita contexts
$(B,B,M_1,N_1,\gamma_1,\delta_1)$ and
$(B,B,M_2,N_2,\gamma_2,\delta_2)$ consists of a couple $(f,g)$
with $f:\ M_1\to M_2$ and $g:\ N_1\to N_2$ left $A^{\sq e}$-module
isomorphisms such that $\gamma_1=\gamma_2\circ (f\ot_B g)$ and
$\delta_1=\delta_2\circ (g\ot_B f)$.

It follows from \prref{5.11} that $\dul{\Pic}^{\sq_H}(B)$ is a
monoidal category, with tensor product induced by the tensor
product over $B$, and unit object $(B,B,B,B,B,B)$. Every object in
$\dul{\Pic}^{\sq_H}(B)$ has an inverse, and we call
$K_0\dul{\Pic}^{\sq_H}(B)= {\Pic}^{\sq_H}(B)$ the $\sq_H$-Picard
group of $B$. From \thref{1.6} and the construction preceding
\prref{5.11}, it follows that $\dul{\Pic}^H(A)$ and
$\dul{\Pic}^{\sq_H}(B)$ are equivalent monoidal categories, so we
conclude that
$
{\Pic}^H(A)\cong {\Pic}^{\sq_H}(B)$.

\section{The $H$-stable part of the Picard group}\selabel{5b}
Throughout this Section, we assume that $H$ is cocommutative.
Now let $A$
be a right $H$-Galois extension of $B$. Our next aim is to
introduce the $H$-invariant subgroup $\Pic(B)^H$ of $\Pic(B)$;
roughly spoken, an object of $\dul{\Pic}(B)$ represents an element
of $\Pic(B)^H$ if its connecting modules $M$ and $N$ are
$H$-stable.
First we need to fix some technical details.\\
We consider
the category ${}_B\Mm_B^H$. Its objects are $B$-bimodules and right $H$-comodules $M$,
such that the right $H$-coaction $\rho$ is left and right $B$-linear, that is,
$\rho(bmb')=bm_{[0]}b'\ot m_{[1]}$,
for all $b,b'\in B$ and $m\in M$. The morphisms are the
$H$-colinear $B$-bimodule maps. For $M,N\in {}_B\Mm_B^H$, we
consider the {\it generalized cotensor product}
\begin{align*} M\ot_B^HN= \{\sum_i m_i\ot_B n_i &\in M\ot_BN \mid  \\
&\sum_i m_{i[0]}\ot_B n_i\ot m_{i[1]}=\sum_i m_i\ot_B n_{i[0]}\ot
n_{i[1]}\}.\end{align*}  Then $M\ot_B^HN$ is an object of
${}_B\Mm_B^H$, with right $H$-coaction
$$\rho(\sum_i m_i\ot n_i)=
\sum_i m_{i[0]}\ot_B n_i\ot m_{i[1]}=\sum_i m_i\ot_B n_{i[0]}\ot n_{i[1]}.$$
We have a functor $-\ot H:\ {}_B\Mm_B\to {}_B\Mm_B^H$. For $M\in {}_B\Mm_B$,
the structure on $M\ot H$ is given by the formulas
$$\rho(m\ot h)=m\ot \Delta(h),~~b(m\ot h)b'=bmb'\ot h.$$
In particular, $B\ot H\in {}_B\Mm_B^H$. The functor $-\ot H$ is monoidal in the sense
of our next Lemma.

\begin{lemma}\lelabel{5.1}
For $M,M'\in {}_B\Mm_B$, we have a natural isomorphism
$$(M\ot H)\ot_B^H (M'\ot H)\cong (M\ot_B M')\ot H$$
in ${}_B\Mm_B^H$.
\end{lemma}

\begin{proof}
It is easy to see that the map
\begin{align*} \kappa:\ (M\ot_B M')\ot H
&\to (M\ot H)\ot_B^H (M'\ot H),\\ m\ot_Bm'\ot h&\mapsto (m\ot
h_{(1)})\ot_B (m'\ot h_{(2)}) \end{align*}
is well-defined and
right $H$-colinear. We claim that $\kappa$ is bijective, with inverse
given by the formula
$$\kappa^{-1}(\sum_j (m_j\ot h_j)\ot_B ( m'_j\ot h'_j))= \sum_j m_j\ot m'_j\ot h_j\varepsilon(h'_j).$$
It is clear that $\kappa^{-1}\circ\kappa=M\ot_B M'\ot H$. If
\[x:=\sum_j (m_j\ot h_j)\ot_B (m'_j\ot h'_j) \in (M\ot H)\ot_B^H (M'\ot
H),\]
then
$$\sum_j (m_j\ot h_{j(1)})\ot_B (m'_j\ot h'_j)\ot h_{j(2)}=\sum_j (m_j\ot h_j)\ot_B (m'_j\ot h'_{j(1)})
\ot h'_{j(2)}.$$
Applying $\varepsilon$ to the third tensor factor, we find
$$(\kappa\circ \kappa^{-1})(x)=\sum_j (m_j\ot h_{j(1)})\ot_B (m'_j\ot \varepsilon(h'_j) h_{j(2)})=x,$$
hence the claim is verified.
\end{proof}

\begin{lemma}\lelabel{5.2}
For all $P\in {}_B\Mm_B^H$, we have that
$$P\ot_B^H(B\ot H)\cong (B\ot H)\ot_B^HP\cong P$$
in ${}_B\Mm_B^H$.
\end{lemma}

\begin{proof}
We have a well-defined morphism
$$\alpha:\ P\to P\ot_B^H(B\ot H),~~\alpha(p)=p_{[0]}\ot_B (1\ot p_{[1]})$$
in ${}_B\Mm_B^H$. The inverse of $\alpha$ is given by the formula
$$\alpha^{-1}(\sum_ip_i\ot_B (b_i\ot h_i))=\sum_i p_ib_i\varepsilon(h_i).$$
It is clear that $\alpha^{-1}\circ\alpha=P$. If
$\sum_ip_i\ot_B (b_i\ot h_i)\in P\ot_B^H(B\ot H)$, then
$$\sum_i p_{i[0]}\ot_B (b_i\ot h_i)\ot p_{i[1]}=
\sum_i p_i\ot_B (b_i\ot h_{i(1)})\ot h_{i(2)}.$$
Then we find
\begin{eqnarray*}
&&\hspace*{-2cm}
(\alpha\circ\alpha^{-1})(\sum_i p_i\ot_B (b_i\ot h_i))=
\sum_i p_{i[0]}b_i\varepsilon(h_i)\ot_B p_{i[1]}\\
&=& \sum_i p_ib_i\varepsilon(h_{i(1)})\ot_B(1\ot h_{i(2)})=\sum_i p_i\ot_B (b_i\ot h_i).
\end{eqnarray*}
\end{proof}

Observe that $A^{\sq e}=A \sq_H A^{{\rm op}} \in {}_{B^e}\Mm_{B^e}^H$, with left and
right $B^e$-action given by the formula
$$(b\ot b')(\sum_i a_i\ot a'_i)(c\ot c')=\sum_i ba_ic\ot c'a'_ib'.$$
Hence we have a second functor
$$A^{\sq e}\ot_{B^e}-:\ {}_B\Mm_B\to {}_B\Mm_B^H.$$
Take $M\in {}_B\Mm_B$. $A^{\sq e}\ot_{B^e}M$ is a left $B^e$-module, and, a fortiori,
a $B$-bimodule. The right $H$-coaction on $A^{\sq e}\ot_{B^e}M$ is given by
the formula
\begin{align*}
 \rho\bigl((\sum_k a_k\ot a'_k)\ot_{B^e} m\bigr)
 &= (\sum_k a_{k[0]}\ot a'_k)\ot_{B^e} m\ot a_{k[1]}    \\
&= (\sum_k a_k\ot a'_{k[0]})\ot_{B^e} m\ot a'_{k[1]}.
\end{align*}
Our next aim is to show that the functor $A^{\sq e}\ot_{B^e}-$ is also monoidal. Before we
can show this, we need a few technical Lemmas.
Let $M\in {}_B\Mm_B$. Then $A^{\rm op}\ot_B M\in {}_B\Mm_B^H$, with the
right $H$-coaction induced by the coaction on $A^{\rm op}$.

\begin{lemma}\lelabel{5.3}
Suppose that $M\in {}_B\Mm_B$ is flat as a left $B$-module. Then the map
$$f:\ A^{\sq e}\ot_B M\to A\sq_H (A^{\rm op}\ot_B M),\quad
\sum_k (a_k\ot a'_k)\ot_B m\mapsto \sum_k a_k\ot (a'_k\ot_B m)$$
is an isomorphism. In a similar way, if $M$ is flat as a right
$B$-module, then
$$M\ot_B A^{\sq e}\cong (M\ot_B A)\sq_H A^{\rm op}.$$
\end{lemma}

\begin{proof}
Consider the commutative diagram
$$\xymatrix{
0\ar[r]&A^{\sq e}\ot_BM\ar[d]_{f}\ar[r]&A^e\ot_B M\ar[d]_{\cong}\ar@<.5ex>[r]\ar@<-.5ex>[r]&
A^e\ot_B M\ot H\\
0\ar[r]&A\sq_H(A^{\rm op}\ot_B M)\ar[r]&A\ot A^{\rm op}\ot_B M\ar@<.5ex>[r]\ar@<-.5ex>[r]&
A^e\ot_B M\ot H}$$
The top row is exact because $M$ is left $B$-flat, and because of the definition of the
generalized cotensor product. The exactness of the bottom row also follows from the
definition of the generalized cotensor product. It follows from the Five Lemma that $f$
is an isomorphism.
\end{proof}

\begin{lemma}\lelabel{5.4}
For all $a\in A$,  the element
$$x:=\sum_i l_i(S(a_{[1]}))\ot_B r_i(S(a_{[1]})) \ot a_{[0]}
\in A\ot_B A^{\sq e}.$$
\end{lemma}

\begin{proof}
By \leref{5.3}, it suffices to show that
$x\in (A\ot_B A)\sq_H A^{\rm op}$. Indeed,
\begin{eqnarray*}
&&\hspace*{-2cm}
\sum_i l_i(S(a_{[1]}))\ot_B r_i(S(a_{[1]}))_{[0]}\ot a_{[0]}\ot r_i(S(a_{[1]}))_{[1]}\\
&\equal{\equref{1.2.3}}&
\sum_i l_i(S(a_{[2]}))\ot_B r_i(S(a_{[2]})) \ot a_{[0]}\ot S(a_{[1]}).
\end{eqnarray*}
\end{proof}

\begin{lemma}\lelabel{5.5}
If $\sum_k a_k\ot a'_k\in A^{\sq e}$, then the element
\begin{eqnarray*}
x&=&\sum_{i,k} a_{k[0]}\ot l_i(a_{k[1]})\ot_B r_i(a_{k[1]})\ot a'_k
\\ &=& \sum_{i,k} a_{k}\ot l_i(S(a'_{k[1]}))\ot_B
r_i(S(a'_{k[1]}))\ot a'_{k[0]}\in A^{\sq
e}\ot_B^H A^{\sq e}.
\end{eqnarray*}
\end{lemma}

\begin{proof}
It follows from \prref{2.7} that $A^{\sq e}$ is flat as a left $B^e$-module.
Since $B^e$ is flat as a left $B$-module, we have that $A^{\sq e}$ is flat as a left $B$-module.
We have shown in \leref{5.4} that $x\in A^{e}\ot_B A^{\sq e}$. Now
\begin{eqnarray*}
&&\hspace*{-2cm}
\sum_{i,k} a_{k[0]}\ot l_i(a_{k[2]})\ot_B r_i(a_{k[2]})\ot a'_k\ot a_{k[1]}\\
&\equal{\equref{1.2.4}}&
\sum_{i,k} a_{k[0]}\ot l_i(a_{k[1]})_{[0]}\ot_B r_i(a_{k[1]})\ot a'_k\ot S( l_i(a_{k[1]})_{[1]}),
\end{eqnarray*}
so $x\in A\sq_H (A^{\rm op}\ot_B A^{\sq e})= A^{\sq e}\ot_B A^{\sq e}$, by \leref{5.3}.
It then follows immediately that $x\in A^{\sq e}\ot_B^H A^{\sq e}$.
\end{proof}

\begin{lemma}\lelabel{5.6}
We have an isomorphism of vector spaces
$f:\ A^{\sq e}\ot B\to A^{\sq e}\ot_B^H A^{\sq e}$,
given by the formula
$$f(\sum_k a_k\ot a'_k\ot b)=
\sum_{i,k} a_{k[0]}\ot bl_i(a_{k[1]})\ot_B r_i(a_{k[1]})\ot a'_k.$$
\end{lemma}

\begin{proof}
It follows from \leref{5.5} that $f$ is well-defined. The inverse of $f$ is defined as follows.
For $y= \sum_k a_k\ot a'_k\ot_B a''_k\ot a'''_k\in A^{\sq e}\ot_B^H A^{\sq e}$, we let
$$f^{-1}(y)=\sum_k a_k\ot a'''_k\ot a'_ka''_k.$$
Let us show that $f^{-1}$ is well-defined. First we show that $f^{-1}(y)\in A^e\ot B$.
Since $y\in A^{\sq e}\ot_B^H A^{\sq e}$, we have that
\begin{eqnarray*}
&&\hspace*{-2cm}
\sum _k a_k\ot a'''_k\ot a'_{k[0]}a''_{k[0]}\ot a'_{k[1]}a''_{k[1]}\\
&=& \sum _k a_k\ot a'''_k\ot a'_{k}a''_{k[0]}\ot S(a''_{k[2]})a''_{k[1]}\\
&=& \sum _k a_k\ot a'''_k\ot a'_{k}a''_k\ot 1.
\end{eqnarray*}
For any vector space $V$, we have that $(V\ot A)^{{\rm co}H}=V\ot B$ ($B$ is flat
over $k$), so the above computation shows that $f^{-1}(y)\in A^e\ot B$.\\
Let us next show that $f^{-1}(y)\in A^{\sq e}\ot B$: since
$y\in A^{\sq e}\ot_B^H A^{\sq e}$, we have that
$$\sum_k a_{k[0]}\ot a'_k\ot_B a''_k\ot a'''_{k}\ot a_{k[1]}=
\sum_k a_{k}\ot a'_k\ot_B a''_k\ot a'''_{k[0]}\ot S(a'''_{k[1]}),$$
hence
$$\sum_k a_{k[0]}\ot  a'''_{k}\ot a'_k a''_k\ot a_{k[1]}=
\sum_k a_{k}\ot a'''_{k[0]}\ot a'_k a''_k\ot S(a'''_{k[1]}).$$
Let us finally verify that $f$ and $f^{-1}$ are inverses.
\begin{eqnarray*}
&&\hspace*{-15mm}
(f^{-1}\circ f)(\sum_k a_k\ot a'_k\ot b)
=f^{-1} \bigl(\sum_{i,k} a_{k[0]}\ot bl_i(a_{k[1]})\ot_B r_i(a_{k[1]})\ot a'_k\bigr)\\
&=&
\sum_{i,k} a_{k[0]}\ot a'_k\ot bl_i(a_{k[1]}) r_i(a_{k[1]})
\equal{\equref{1.2.5}} \sum_k a_k\ot a'_k\ot b;
\end{eqnarray*}
\begin{eqnarray*}
&&\hspace*{-15mm}
(f\circ f^{-1})(\sum_k a_k\ot a'_k\ot_B a''_k\ot a'''_k)
=f(\sum_k a_k\ot a'''_k\ot a'_ka''_k)\\
&=&
\sum_{i,k} a_{k[0]}\ot a'_ka''_kl_i(a_{k[1]})\ot_B r_i(a_{k[1]})\ot a'''_k\\
&=&
\sum_{i,k} a_k\ot a'_ka''_{k[0]}l_i(a''_{k[1]})\ot_B r_i(a''_{k[1]})\ot a'''_k\\
&\equal{\equref{1.2.6}}&\sum_k a_k\ot a'_k\ot_B a''_k\ot a'''_k.
\end{eqnarray*}
\end{proof}

Take $M,M'\in {}_B\Mm_B$ and consider the composition
$\tilde{g}=({\rm id}\ot \can^{-1}\ot {\rm id}) \circ (\rho_A\ot {\rm id})$:
\begin{eqnarray*}
&&\hspace*{-2cm}
A^e\ot_{B^e}(M\ot_B M')\cong A\ot_B M\ot_BB\ot_B M'\ot_B A\\
&\longrightarrow& A\ot_B M\ot_BB\ot H\ot_B M'\ot_B A\\
&\longrightarrow& A\ot_B M\ot_BA\ot_BA\ot_B M'\ot_B A\\
&\cong& A^e\ot_{B^e} M\ot_B A^{e}\ot_{B^e}M'.
\end{eqnarray*}
We compute that
\begin{eqnarray*}
&&\hspace*{-2cm}
\tilde{g}(\sum_k a_k\ot a'_k\ot_{B^e}(m\ot_B m'))\\
&=&\sum_{i,k} (a_{k[0]}\ot l_i(a_{k[1]}))\ot_{B^e} m\ot_B
(r_i(a_{k[1]})\ot a'_k)\ot_{B^e} m'.
\end{eqnarray*}
It follows from \leref{5.6} that $\tilde{g}$ restricts to a map
$$g:\ A^{\sq e}\ot_{B^e}(M\ot_B M')\to (A^{\sq e}\ot_{B^e}M)\ot_B^H (A^{\sq e}\ot_{B^e}M').$$
It is obvious that $g\in {}_B\Mm_B^H$, and that $g$ is bijective with inverse
\begin{eqnarray}
&&\hspace*{-2cm}\eqlabel{5.7.1}
g^{-1}\bigl(\sum_k (a_k\ot a'_k)\ot_{B^e}m\ot_B (a''_k\ot a'''_k)\ot_{B^e}m'\bigr)\\
&=& \sum_k (a_k\ot a'''_k)\ot_{B^e}(ma'_ka''_k\ot_{B^e} m')\nonumber\\
&=& \sum_k (a_k\ot a'''_k)\ot_{B^e}(m\ot_{B^e} a'_ka''_k m').\nonumber
\end{eqnarray}
As a conclusion, we obtain the following Lemma.

\begin{lemma}\lelabel{5.7}
For $M,M'\in {}_B\Mm_B$, we have an isomorphism
$$g:\ A^{\sq e}\ot_{B^e}(M\ot_B M')\to (A^{\sq e}\ot_{B^e}M)\ot_B^H (A^{\sq e}\ot_{B^e}M').$$
\end{lemma}

\begin{remark}\relabel{5.8}
It follows from Lemmas \ref{le:5.1} and \ref{le:5.7} that, for $M,M',M''\in {}_B\Mm_B$, we
have isomorphisms
\begin{eqnarray*}
&&\hspace*{-2cm}
\bigl((M\ot H)\ot_B^H (M'\ot H)\bigr)\ot_B^H (M''\ot H)\\
&\cong& (M\ot H)\ot_B^H \bigl((M'\ot H)\ot_B^H (M''\ot H)\bigr),\\
&&\hspace*{-2cm}
\bigl((A^{\sq e}\ot_{B^e}M)\ot_B^H  (A^{\sq e}\ot_{B^e}M')\bigr)
\ot_B^H  (A^{\sq e}\ot_{B^e}M'')\\
&\cong &
(A^{\sq e}\ot_{B^e}M)\ot_B^H \bigl( (A^{\sq e}\ot_{B^e}M')
\ot_B^H  (A^{\sq e}\ot_{B^e}M'')\bigr)
\end{eqnarray*}
in ${}_B\Mm_B^H$ that are natural in $M,M',M''$.
\end{remark}

We now consider the notion of $H$-stability, as introduced before
\coref{4.10}, but with $B$ replaced by $B^e$ and $A$ by $A^{\sq
e}$. The $(B,B)$-bimodule $M$ is $H$-stable if there exists an
isomorphism \[\varphi_M:\ A^{\sq e}\ot_{B^e} M\to M\ot H\] in the
category ${}_B\Mm_B^H$.

\begin{proposition}\prlabel{5.9}
If $M,M'\in {}_B\Mm_B$ are $H$-stable, then $M\ot_B M'$ is also $H$-stable.
\end{proposition}

\begin{proof}
We define $\varphi_{M\ot_B M'}$ by the commutativity of the following diagram:
\begin{equation}\eqlabel{5.9.1}
\xymatrix{
A^{\sq e}\ot_{B^e} (M\ot_B M')\ar[rr]^{g}\ar[d]_{\varphi_{M\ot_B M'}}
&& (A^{\sq e}\ot_{B^e} M)\ot_B^H (A^{\sq e}\ot_{B^e} M')
\ar[d]^{\varphi_M\ot_B^H \varphi_{M'}}\\
M\ot_BM'\ot H\ar[rr]^{\kappa}&&(M\ot H)\ot_B^H(M'\ot H)}
\end{equation}
\end{proof}

Suppose that $M,M'\in {}_B\Mm_B$ are $H$-stable, and let $\psi_M=
\varphi_M^{-1}$, $\psi_{M'}=\varphi_{M'}^{-1}$, $t_M=\alpha_{\bf 12}^{-1}(\psi_M)$,
$t_{M'}=\alpha_{\bf 12}^{-1}(\psi_{M'})$. For later use, we compute
$t_{M\ot_B M'}=\alpha_{\bf 12}^{-1}(\psi_{M\ot_B M'})$ in terms of $t_M$ and $t_{M'}$.
To this end, we first introduce the following Sweedler-type notation for the map $t_M$:
$$t_M(h)(1_{A^{\sq \rm e}}\ot_{B^e} m)=(m(h)^+\ot m(h)^-)\ot_{B^e} m(h)^0.$$
Summation is implicitly understood.
Using the definition of $\alpha_{\bf 12}$ and the commutativity of
\equref{5.9.1}, we compute
\begin{align*}
 t&{}_{M\ot_B M'}(h)(1_{A^{\sq \rm e}}\ot_{B^e} (m\ot_Bm'))
=\psi_{M\ot_B M'}(m\ot_Bm'\ot h)\\
&=
\bigl( g^{-1}\circ (\psi_M\ot_B^H \psi_{M'})\circ k \bigr)(m\ot_Bm'\ot h)\\
&=
\bigl( g^{-1}\circ (\psi_M\ot_B^H \psi_{M'})\bigr)\bigl((m\ot h_{(1)})\ot_B (m'\ot h_{(2)})\bigr)\\
&= g^{-1}\bigl(
t(h_{(1)})(1_{A^{\sq \rm e}}\ot_{B^e} m)\ot_B t(h_{(2)})(1_{A^{\sq \rm e}}\ot_{B^e} m')\bigr)\\
&=
g^{-1}\Bigl(\bigl( (m(h_{(1)})^+\ot m(h_{(1)})^-)\ot_{B^e} m(h_{(1)})^0\bigr)\\
&\hspace*{1cm}\ot_B
\bigl( (m'(h_{(2)})^+\ot m'(h_{(2)})^-)\ot_{B^e} m'(h_{(2)})^0\bigr)\Bigr)\\
&= (m(h_{(1)})^+\ot m'(h_{(2)})^-)\ot_{B^e}
(m(h_{(1)})^0m(h_{(1)})^-m'(h_{(2)})^+\ot_B m'(h_{(1)})^0)\\
&= (m(h_{(1)})^+\ot m'(h_{(2)})^-)\ot_{B^e}
(m(h_{(1)})^0\ot_Bm(h_{(1)})^-m'(h_{(2)})^+ m'(h_{(1)})^0).
\end{align*}
We will need a slight improvement of this formula. For $b\in B$, we have
\begin{eqnarray*}
&&\hspace*{-15mm}
t_M(h)(1_{A^{\sq \rm e}}\ot_{B^e} mb) =t_M(h)((1\ot b)\ot_{B^e} m)\\
&=& (1\ot b)t_M(h)(1_{A^{\sq \rm e}}\ot_{B^e} m) = (m(h)^+\ot
m(h)^-b)\ot_{B^e} m(h)^0,
\end{eqnarray*}
hence
\begin{eqnarray}
\eqlabel{5.9.2}
&&\hspace*{-1cm}
t_{M\ot_B M'}(h)(1_{A^{\sq \rm e}}\ot_{B^e} (mb\ot_Bm'))
=
(m(h_{(1)})^+\ot m'(h_{(2)})^-)\\
&&\hspace*{1cm}\ot_{B^e}
(m(h_{(1)})^0m(h_{(1)})^-bm'(h_{(2)})^+\ot_B m'(h_{(1)})^0).\nonumber
\end{eqnarray}

We have that $B$ is a left $A^{\sq e}$-module, with action
$\phi_B((\sum_k a_k\ot a'_k)\ot_{B^e} b) = \sum_k a_kba'_k$. The
corresponding map $\varphi_B=\beta_1(\phi_B):\ A^{\sq e}\ot_{B^e} B\to B\ot H$
is given by
$$\varphi_B((\sum_k a_k\ot a'_k)\ot_{B^e} b)
=\sum_k a_{k[0]}ba'_k\ot a_{k[1]}= \sum_k a_kba'_{k[0]}\ot
S(a_{k[1]}).$$ It follows from \coref{4.10} and \prref{4.11} that
$\varphi_B$ is an isomorphism in ${}_B\Mm_B^H$.

Now take $\ul{M}=(B,B,M,N,\alpha,\beta)\in \dul{\Pic}(B)$. We call $\ul{M}$
$H$-stable if there exist isomorphisms
$$\varphi_M:\ A^{\sq e}\ot_{B^e} M\to M\ot H\quad {\rm and} \quad
\varphi_N:\ A^{\sq e}\ot_{B^e} N\to N\ot H$$
such that the following diagrams commute:
\begin{equation}\eqlabel{5.9.1a}
\xymatrix{
A^{\sq e}\ot_{B^e} (M\ot_B N)\ar[rr]^{A^{\sq e}\ot_{B^e}\alpha}
\ar[d]_{\varphi_{M\ot_B N}}&&
A^{\sq e}\ot_{B^e} B \ar[d]^{\varphi_B}\\
M\ot_BN\ot H\ar[rr]^{\alpha\ot H} &&B\ot H}
\end{equation}
\begin{equation}\eqlabel{5.9.2a}
\xymatrix{
A^{\sq e}\ot_{B^e} (N\ot_B M)\ar[rr]^{A^{\sq e}\ot_{B^e}\beta}
\ar[d]_{\varphi_{N\ot_B M}}&&
A^{\sq e}\ot_{B^e} B \ar[d]^{\varphi_B}\\
N\ot_BM\ot H\ar[rr]^{\beta\ot H} &&B\ot H}
\end{equation}

\begin{theorem}\thlabel{5.10}
Let $H$ be a cocommutative Hopf algebra, and let $A$ be a faithfully flat
Hopf-Galois extension of $A^{{\rm co}H}=B$. Then
$$\Pic(B)^H=\{[\ul{M}]\in \Pic(B)~|~\ul{M}~{\rm is~}H\hbox{-}{\rm stable}\}$$
is a subgroup of $\Pic(B)$, called the $H$-stable part of
$\Pic(B)$.
\end{theorem}

\begin{proof}
Assume that $\ul{M}_1$ and $\ul{M}_2$ are $H$-stable. It follows from \prref{5.9}
that $M_1\ot_B M_2$ and $N_2\ot_B N_1$ are $H$-stable. A commutative diagram
argument taking \reref{5.8} into account shows that the diagrams (\ref{eq:5.9.1a}-\ref{eq:5.9.2a}),
with $M$ replaced by $M_1\ot_B M_2$ and $N$ by $N_2\ot_B N_1$, commute.
This implies that $\ul{M}_1\ot_B\ul{M}_2$ is $H$-stable. Finally, if $\ul{M}$ is
$H$-stable, then it is clear from the definition that $\ul{M}^{-1}=
(B,B,N,M,\beta,\alpha)$ is also $H$-stable.
\end{proof}

\section{A Hopf algebra version of the Beattie-del R\'{\i}o \\ exact sequence}\selabel{6}
As in the previous Section, let $H$ be a cocommutative Hopf algebra, and $A$
a faithfully flat $H$-Galois extension of $B$. Take $\ul{M}\in \dul{\Pic}(B)^H$.
Then we have an isomorphism
$\varphi:\ A^{\sq e}\ot_{B^e} M\to M\ot H$
in ${}_B\Mm_B^H$. We have that
$E={}_{A^{\sq e}}\END(A^{\sq e}\ot_{B^e}M)^{\rm op}$
is an $H$-comodule algebra.

\begin{lemma}\lelabel{6.1}
$E^{{\rm co}H}\cong Z(B)$.
\end{lemma}

\begin{proof}
We first observe that
$$E^{{\rm co}H}= {}_{A^{\sq e}}\End^H(A^{\sq e}\ot_{B^e}M)
\cong {}_{B^e}\End(M)={}_B\End_B(M).$$
The second isomorphism is due to the fact that
$A^{\sq e}\ot_{B^e}-:\ {}_{B^e}\Mm\to {}_{A^{\sq e}}\Mm^H$
is a category equivalence, by \thref{1.1} and \prref{2.7}. Since $\ul{M}$
is a strict Morita context, we have that
$-\ot_B M$ is an autoequivalence of $\Mm_B$.
$-\ot_B M$ and its adjoint send $B$-bimodules to
$B$-bimodules, so $-\ot_B M$ also defines an autoequivalence of ${}_B\Mm_B$.
Consequently ${}_B\End_B(M)\cong {}_B\End_B(B)\cong Z(B)$.
\end{proof}

For later use, we give an explicit description of the isomorphism
$$\lambda:\ Z(B)\to E^{{\rm co}H}= {}_{A^{\sq e}}\End^H(A^{\sq e}\ot_{B^e}M),
~~~x\mapsto \lambda_x:$$
\begin{equation}\eqlabel{6.1.1}
\lambda_x(\sum_k(a_k\ot a'_k)\ot_{B^e}m)=
\sum_k(a_k\ot a'_k)\ot_{B^e}xm.
\end{equation}

We have seen in \thref{4.5} that there are isomorphisms
$$\alpha_{\bf 12}:\ \Hom^H(H,E)\to {}_B\Hom_B^H(M\ot H, A^{\sq e}\ot_{B^e} M),$$
$$\alpha_{\bf 21}:\ \Cc({\bf 2},{\bf 1})\to  {}_B\Hom_B^H(A^{\sq e}\ot_{B^e} M, M\ot H).$$
Using \prref{2.7}, we compute $u=\alpha_{\bf 21}^{-1}(\varphi)$ and
$t=\alpha_{\bf 12}^{-1}(\varphi^{-1})$:
\begin{equation}\eqlabel{6.1.2}
t(h)\bigl( \sum_k(a_k\ot a'_k)\ot_{B^e}m\bigr)
= \sum_k(a_k\ot a'_k)\varphi^{-1}(m\ot h);
\end{equation}
\begin{eqnarray}\eqlabel{6.1.3}
&&\hspace*{-15mm}
u(h)\bigl( \sum_k(a_k\ot a'_k)\ot_{B^e}m\bigr)
=\sum_{i,j,k}
\bigl( a_kl_i(h_{(1)})\ot r_j(h_{(2)})a'_k\bigr)\\
&&~~~~~\ot_{B^e}
\phi \bigl( (r_i(h_{(1)})\ot l_j(h_{(2)}))\ot_{B^e} m\bigr).\nonumber
\end{eqnarray}
Since $E^{{\rm co}H}\cong Z(B)$ is commutative, we can apply
\prref{4.2}, and we find that $Z(B)$ is a left $H$-module algebra.
We will show in \prref{6.4} that the left $H$-action on $Z(B)$ is
independent of the choice of $\ul{M}\in \dul{\Pic}(B)^H$, and is
given by the Miyashita-Ulbrich action \equref{1.2.8}.

\begin{lemma}\lelabel{6.3}
For $x\in Z(B)$, $m\in M$ and $h\in H$, we have that
\begin{equation}\eqlabel{6.3.1}
\lambda_x(\varphi^{-1}(m\ot h))=\varphi^{-1}((h_{(2)}\cdot x)m\ot h_{(1)}).
\end{equation}
\end{lemma}

\begin{proof}
Write
\begin{equation}\eqlabel{6.3.2}
\varphi^{-1}(m\ot h)=\sum_k (s_k\ot s'_k)\ot_{B^e} m_k\in A^{\sq e}\ot_{B^e}M.
\end{equation}
Since $\varphi^{-1}$ is right $H$-colinear, we have that
\begin{equation}\eqlabel{6.3.3}
\varphi^{-1}(m\ot h_{(1)})\ot h_{(2)}=\sum_k (s_{k[0]}\ot s'_k)\ot_{B^e} m_k\ot s_{k[1]}.
\end{equation}
Then we compute
\begin{eqnarray*}
&&\hspace*{-15mm}
\varphi^{-1}(xm\ot h)=x\varphi^{-1}(m\ot h)\\
&\equal{\equref{6.3.2}}& \sum_k (xs_k\ot s'_k)\ot_{B^e} m_k
\equal{\equref{1.2.9}}
\sum_k (s_{k[0]}(S(s_{k[1]})\cdot x)\ot s'_k)\ot_{B^e} m_k\\
&=&
\sum_k (s_{k[0]}\ot s'_k)\ot_{B^e} (S(s_{k[1]})\cdot x)m_k
\equal{(\ref{eq:6.1.1},\ref{eq:6.3.3})}
\lambda_{S(h_{(2)})\cdot x}(\varphi^{-1}(m\ot h_{(1)}).
\end{eqnarray*}
and it follows that
$$\varphi^{-1}((h_{(2)}\cdot x)m\ot h_{(1)})=
\lambda_{S(h_{(2)})\cdot (h_{(3)}\cdot x)}(\varphi^{-1}(m\ot h_{(1)}))=
\lambda_x(\varphi^{-1}(m\ot h)).$$
\end{proof}

\begin{proposition}\prlabel{6.4}
Assume that $\ul{M}\in \dul{\Pic}(B)$ is $H$-stable. The corresponding left
$H$-action on $E^{{\rm co}H}$ is given by the formula
$h\bullet \lambda_x=\lambda_{h\cdot x}$,
for all $x\in Z(B)$. This means that the transported action on
$Z(B)$ is the Miyashita-Ulbrich action given by
{\rm\equref{1.2.9}}.
\end{proposition}

\begin{proof}
Take $x\in Z(B)$ and the corresponding $\lambda_x\in E^{{\rm co}H}$. The
action of $h\in H$ on $\lambda_x$ is given by (see \prref{4.2})
$$h\bullet \lambda_x=u(h_{(1)})\circ \lambda_x\circ t(h_{(2)}),$$
and we have
\begin{eqnarray*}
\diamond&:=& (h\bullet \lambda_x)(\sum_k (a_k\ot a'_k)\ot_{B^e} m))\\
&\equal{\equref{6.1.2}}&
(u(h_{(1)})\circ \lambda_x)\bigl(\sum_k (a_k\ot a'_k)\varphi^{-1}((h_{(3)}\cdot x)m\ot h_{(2)}))\\
&\equal{\equref{6.3.1}}&
u(h_{(1)})\bigl(\sum_k (a_k\ot a'_k)\varphi^{-1}((h_{(3)}\cdot x)m\ot h_{(2)})\bigr).
\end{eqnarray*}
Now write
\begin{equation}\eqlabel{6.4.1}
\varphi^{-1}((h_{(2)}\cdot x)m\ot h_{(1)})=
\sum_q (s_q\ot s'_q)\ot_{B^e} m_q.
\end{equation}
Since $\varphi^{-1}$ is right $H$-colinear, we have that
$$
\varphi^{-1}((h_{(3)}\cdot x)m\ot h_{(2)})\ot h_{(1)}=
\sum_q (s_{q[0]}\ot s'_q)\ot_{B^e} m_q\ot s_{q[1]},
$$
hence
\begin{align*}
&
\diamond= u(s_{q[1]})\bigl( \sum_{q,k} (a_ks_{q[0]}\ot s'_qa'_k)\ot_{B^e} m_q)\\
&\equal{\equref{6.1.3}}
\sum_{i,j,k,q}\bigl( q_ks_{q[0]}l_i(s_{q[1]})\ot r_j(s_{q[2]})s'_qa'_k\bigr)
\ot_{B^e}
\phi\bigl((r_i(s_{q[1]})\ot l_j(s_{q[2]}))\ot_{B^e}m_q.
\end{align*}
Using (\ref{eq:1.2.6},\ref{eq:2.7.1}), we find
\begin{eqnarray*}
&&\hspace*{-2cm}
\sum_{i,q} (s_{q[0]}\ot s'_q)(l_i(s_{q[1]})\ot r_j(s_{q[2]}))\ot_{B^e}
(r_i(s_{q[1]})\ot l_j(s_{q[2]}))\\
&=&\sum_q 1_{A^{\sq e}}\ot_{B^e} (s_q\ot s'_q)\in B^e\ot_{B^e} A^{\sq e}.
\end{eqnarray*}
Since $\phi$ is left $B^e$-linear, we find
\begin{eqnarray*}
\diamond&=& \sum_{i,j,k,q} \bigl(a_k\ot a'_k)\ot_{B^e}\\
&&\hspace*{1cm}
\phi\bigl((s_{q[0]}l_i(s_{q[1]})r_i(s_{q[1]})\ot l_j(s_{q[2]})r_j(s_{q[2]})s'_q)\ot_{B^e} m_q\\
&=& \sum_{k,q} \bigl(a_k\ot a'_k)\ot_{B^e}
\phi\bigl(s_q\ot s'_q)\ot_{B^e} m_q\bigr)\\
&\equal{\equref{6.4.1}}& \sum_{k,q} \bigl(a_k\ot a'_k)\ot_{B^e} ((M\ot\varepsilon)
\circ\varphi\circ\varphi^{-1})((h_{(2)}\cdot x)m\ot h_{(1)})\\
&=& \sum_{k,q} \bigl(a_k\ot a'_k)\ot_{B^e} (h\cdot x)m
= \lambda_{h\cdot x}\bigl( \sum_k(a_k\ot a'_k)\ot_{B^e} m\bigr).
\end{eqnarray*}
This shows that $h\bullet \lambda_x=\lambda_{h\cdot x}$, for all $x\in Z(B)$.
\end{proof}

It follows from the discussion in \seref{5} that the functor
$\dul{\Pic}^{\sq_H}(B)\to \dul{\Pic}(B)$ restricting the $A^{\sq
e}$-module structure on the connecting bimodules to the
$B$-bimodule structure is strongly monoidal. This implies that we
have a group homomorphism \[g_2:\ {\Pic}^{\sq_H}(B)\to
{\Pic}(B).\]

\begin{proposition}\prlabel{6.5} The groups
$\Ker(g_2)$ and $H^1(H,Z(B))$ are isomorphic.
\end{proposition}

\begin{proof}
Take $[\ul{M}]=[(B,B,M,N,\alpha,\beta)]\in \Ker(g_2)$. Then $M$
and $N$ are isomorphic to $B$ as $B$-bimodules. $\ul{M}$ is
described completely once we know the left $A^{\sq e}$-module
structure on $M=B$, by \thref{1.6} (2). Isomorphism classes of
left $A^{\sq e}$-module structures on $B$ are in bijective
correspondence to the elements of $\ol{\Omega}_E$, cf.
\coref{4.13}. It follows from \prref{4.3} that $\ol{\Omega}_E\cong
H^1(H,Z(B))$, hence we have a bijection between $H^1(H,Z(B))$ and
$\Ker(g_2)$, and an injection \[g_1:\ H^1(H,Z(B))\to
{\Pic}^{\sq_H}(B).\] We will now describe this injection
explicitly, and show that it preserves multiplication.\\
Let $\phi_0$ be the left $A^{\sq e}$-action on $B$ corresponding to the trivial
element in ${\Pic}^{\sq_H}(B)$:
$$\phi_0\bigl(\sum_k(a_k\ot a'_k)\ot_{B^e} b\bigr)=\sum_k a_kba'_k.$$
Let $u_0=\tilde{\alpha}_{\bf 21}^{-1}$ be the corresponding element in
$\Cc_E({\bf 2},{\bf 1})$. Using the formulas in the proof of \thref{4.5} we obtain that
$$u_0(h)\bigl(\sum_k(a_k\ot a'_k)\ot_{B^e} b\bigr)=
\sum_{i,j,k} (a_kl_i(h_{(1)})\ot r_j(h_{(2)})a'_k)\ot_{B^e}
r_i(h_{(1)})bl_j(h_{(2)}).$$ Let $\alpha\in Z^1(H,Z(B))$, and take
$G(\alpha)=t=\alpha*t_0\in \Omega_E$ (see \prref{4.3}). Then
$t(h)=t_0(h_{(1)})\circ \alpha(h_{(2)})$, and
$u(h)=t(S(h))=u_0(h_{(1)})\circ \alpha(S(h_{(2)}))$.
We compute $\phi_\alpha=\tilde{\alpha}_{\bf 21}(u)$, using the formulas given in
the proof of \thref{4.5}:
\begin{eqnarray*}
&&\hspace*{-15mm}
1\ot_{B^e} \phi_\alpha\bigl(\sum_k (a_k\ot a'_k)\ot_{B^e} b\bigr)\\
&=& \sum_k u(a_{k[1]})\bigl((a_{k[0]}\ot a'_k)\ot_{B^e} b\bigr)\\
&=& \sum_k u_0(a_{k[1]})\bigl((a_{k[0]}\ot a'_k)\ot_{B^e} \alpha(S(a_{k[2]}))b\bigr)\\
&=& \sum_{i,j,k}\bigl(a_{k[0]}l_i(a_{k[1]})\ot r_j(a_{k[2]})a'_k\bigr)\ot_{B^e}
r_i(a_{k[1]})\alpha(S(a_{k[3]}))bl_j(a_{k[2]})\\
&=& \sum_{i,j,k}1_{A^{\sq e}}\ot_{B^e} a_{k[0]}l_i(a_{k[1]})r_i(a_{k[1]})\alpha(S(a_{k[3]}))bl_j(a_{k[2]})r_j(a_{k[2]})a'_k\\
&=& \sum_k 1_{A^{\sq e}}\ot_{B^e} a_{k[0]}\alpha(S(a_{k[1]}))ba'_k
\end{eqnarray*}
This means that $g_1(\alpha)$ is represented by $B$, with left $A^{\sq e}$-action given by
\begin{equation}\eqlabel{6.5.1}
\sum_k (a_k\ot a'_k)\cdot_\alpha b=
\phi_\alpha\bigl(\sum_k (a_k\ot a'_k)\ot_{B^e} b\bigr)=a_{k[0]}\alpha(S(a_{k[1]}))ba'_k.
\end{equation}
Let $\beta\in Z^1(H,Z(B))$ be another cocycle. Then
$g_1(\alpha)\ot_B g_1(\beta)=B\ot_BB\cong B$ as a
$(B,B)$-bimodule, with left $A^{\sq e}$-action
\begin{eqnarray*}
&&\hspace*{-15mm}
\sum_k (a_k\ot a'_k)\cdot b\cong \sum_k (a_k\ot a'_k)\cdot (1\ot_B b)\\
&\equal{\equref{5.11.1}}&
\sum_{i,k} (a_{k[0]}\ot l_i(a_{k[1]}))\cdot_\alpha 1\ot_B (r_i(a_{k[1]})\ot a'_k)\cdot_{\beta}b\\
&\equal{\equref{6.5.1}}&
\sum_{i,k} a_{k[0]}\alpha(S(a_{k[1]}))l_i(a_{k[2]})\ot_B r_i(a_{k[2]})_{[0]}
\beta(S(r_i(a_{k[2]})_{[1]}))ba'_k\\
&\equal{\equref{1.2.3}}&
\sum_{i,k} a_{k[0]}\alpha(S(a_{k[1]}))l_i(a_{k[2]})\ot_B r_i(a_{k[2]}) \beta(S(a_{k[3]}))ba'_k\\
&\cong & \sum_{i,k} a_{k[0]}\alpha(S(a_{k[1]}))l_i(a_{k[2]}) r_i(a_{k[2]}) \beta(S(a_{k[3]}))ba'_k\\
&\equal{\equref{1.2.5}}&
\sum_{k} a_{k[0]}\alpha(S(a_{k[1]})) \beta(S(a_{k[2]}))ba'_k\\
&=& \sum_{k} a_{k[0]}(\alpha*\beta)(S(a_{k[1]})) ba'_k\\
&=& \sum_k (a_k\ot a'_k)\cdot_{\alpha*\beta} b.
\end{eqnarray*}
This shows that $g_1(\alpha)\ot_B g_1(\beta)=g_1(\alpha*\beta)$,
that is,  $g_1$ is a group monomorphism.
\end{proof}

Let $\ul{M}\in \dul{\Pic}(B)$ be $H$-stable. Then there exists an
isomorphism \[\psi:\ M\ot H\to A^{\sq e}\ot_{B^e} M\] in
${}_B\Mm_B^H$ such that $\psi(m\ot 1)=1_{A^{\sq e}}\ot_{B^e}m$,
for all $m\in M$ (see the arguments given after \coref{4.10}). Then
$t:=\alpha_1^{-1}(\psi)\in \Hom^H(H,E)$ is convolution invertible
and satisfies the condition $t(1)=1$. In \prref{4.4}, we
constructed a cocycle $\sigma\in Z^2(H,Z(B))$. Now let
$g_3([\ul{M}])=[\sigma]\in H^2(H,Z(B))$. This defines a map
$$g_3:\ \Pic(B)^H\to  H^2(H,Z(B)).$$
It follows from \prref{4.4} that $g_3([\ul{M}])=1$ if and only if there exists an algebra
map $t'\in \Hom^H(H,E)$. By \prref{4.11}, this is equivalent to the existence of an
associative left $A^{\sq e}$-action $\phi:\ A^{\sq e}\ot_{B^e} M\to M$, which is
equivalent to $[\ul{M}]\in \im(g_2)$. We conclude that $\im(g_2)=\Ker(g_3)$.
Our observations can be summarized as follows.

\begin{theorem}\thlabel{6.6}
Let $H$ be a cocommutative Hopf algebra over a field $k$, and $A$ a faithfully flat
Hopf-Galois extension of $B=A^{{\rm co}H}$. Then we have an exact sequence
$$
1\to H^1(H,Z(B))\overset{g_1}\to \Pic^{\sq_H}(B)\cong\Pic^H(A)
 \overset{g_2}\to \Pic(B)^H \overset{g_3}\to H^2(H,Z(B)).$$
\end{theorem}

Observe that $\Pic^{\sq_H}(B)\cong\Pic^H(A)$ and $\Pic(B)^H$ are
non-abelian groups. The category of groups is not an abelian
category, so it makes no sense to talk about exact sequences of
groups. In the  statement in \thref{6.6}, exactness  means that
$g_1$ is an injective map, and that
$\im(g_i)=\{x~|~g_{i+1}(x)=1\}$, for $i=1,2$. The maps $g_1$ and
$g_2$ are group homomorphisms. An example given in \cite{BR2}
shows that $g_3$ is not a group homomorphism in general, even in
the case of group graded algebras. We will discuss in \seref{7}
the property satisfied by $g_3$.

\section{$g_3$ is a $1$-cocycle}\selabel{7}
We recall from \cite{F} that $\Pic(B)$ acts on $Z(B)$ as follows. For $[\ul{M}]\in \Pic(B)$,
we have a map $\xi_M:\ Z(B)\to Z(B)$ characterized by the property
\begin{equation}\eqlabel{7.1.0}
\xi_M(x)=y~~\Longleftrightarrow~~mx=ym,\quad{\rm for~all}~m\in M.
\end{equation}
It is easy to show that $\xi_M(xy)=\xi_M(x)\xi_M(y)$. We will show
that this action defines an action of $\Pic(B)^H$ on
$H^n(H,Z(B))$, so that we can consider the group of cocycles
$Z^1(\Pic(B)^H, H^2(H,Z(B)))$.
We will then show that $g_3$ is such a $1$-cocycle.\\
Our first aim is to show that the action $\Pic(B)^H$ on $Z(B)$ commutes with the
action of $H$ on $Z(B)$. First, we need some Lemmas.

\begin{lemma}\lelabel{7.1}
Take $[\ul{M}]\in \Pic(B)^H$. For all $x\in Z(B)$, $m\in M$ and $\sum_i a_i\ot a'_i\in A^{\sq e}$,
we have that
\begin{equation}\eqlabel{7.1.1}
\bigl(\sum_i a_i\ot a'_ix\bigr)\ot_{B^e} m= \bigl(\sum_i \xi_M(x)a_i\ot a'_i\bigr)\ot_{B^e} m
\end{equation}
in $A^{\sq e}\ot_{B^e} M$.
\end{lemma}

\begin{proof}
This follows immediately from the fact that $mx\ot h=\xi_M(x)m\ot
h$ in $M\ot H$, for all $m\in M$, $x\in Z(B)$ and $h\in H$, and
the fact that we a $(B,B)$-bimodule isomorphism $\psi_M:\ M\ot
H\to A^{\sq e}\ot_{B^e} M$.
\end{proof}

\begin{lemma}\lelabel{7.2}
The map \[l:\ A^{\sq e}\ot_{B^e} B\to A\ot_B^HA^{\rm op},\quad
(\sum_i a_i\ot a'_i)\ot_{B^e}b\mapsto \sum_i a_ib\ot_B a'_i\] is
an isomorphism.
\end{lemma}

\begin{proof}
Observe first that $A^{\sq e}\ot_{B^e} B$ and $A\ot_B^HA^{\rm op}$ are objects of the
category ${}_{A^{\sq e}}\Mm^H$. It follows from \thref{1.1} and \prref{2.7} that it suffices
to show that
$$(A\ot_B^HA^{\rm op})^{{\rm co}H}\cong B\cong (A^{\sq e}\ot_{B^e} B)^{{\rm co}H}.$$
Take
$$\sum_i a_i\ot_B a'_i\in (A\ot_B^HA^{\rm op})^{{\rm co}H}\subset
A\ot_B^HA^{\rm op}.$$
Then
$$\sum_i a_{i[0]}\ot_B a'_i\ot a_{i[1]}=\sum_i (a_ib\ot_B a'_i) \ot 1.$$
From the fact that $A\in {}_B\Mm$ is faithfully flat, we deduce
that $\sum_i a_i\ot_B a'_i\in A^{{\rm co}H}\ot_B A=B\ot_B A$,
hence
$$\sum_i a_i\ot_B a'_i=1\ot_B\sum_i a_i a'_i=1\ot_B a.$$
Since $\sum_i a_i\ot_B a'_i\in A\ot_B^HA^{\rm op}$, we also have that
$$1\ot_B a_{[0]}\ot S(a_{[1]})=1\ot_B a\ot 1.$$
Apply $\rho_A$ to the second tensor factor ($\rho_A$ is left $B$-linear), and then
multiply the second and third tensor factor. This gives
$1\ot_B a_{[0]}\ot a_{[1]}=1\ot_B a\ot_B 1$, and it follows that $a\in B$. This shows that
the map
$$f:\ B\to (A\ot_B^HA^{\rm op})^{{\rm co}H},\quad f(b)=1\ot_B b$$
is an isomorphism.
\end{proof}

\leref{7.2} tells us that the map $A^{\sq e}\to A\ot_B^HA^{\rm op}$ induced by the
canonical surjection $A^e\to A\ot_BA^{\rm op}$ is surjective.

\begin{proposition}\prlabel{7.3}
Let $\ul{M}=(B,B,M,N,\alpha,\beta)$ represent an $H$-stable element of $\Pic(B)$.
Then
$$\xi_M(h\cdot x)=h\cdot (\xi_M(x)),$$
for all $h\in H$ and $x\in Z(B)$.
\end{proposition}

\begin{proof}
For $\sum_k a_k\ot a'_k\in A^{\sq e}$, $x\in Z(B)$ and $m\in M$, we compute that
\begin{eqnarray*}
&&\hspace*{-25mm}
(\sum_k \xi_M(x)a_k\ot a'_k)\ot_{B^e} m
~\equal{\equref{7.1.1}}~
(\sum_k a_k\ot a'_kx)\ot_{B^e} m\\
&\equal{\equref{1.2.9}}&
(\sum_k a_k\ot (a'_{k[1]}\cdot x)a'_{k[0]})\ot_{B^e} m\\
&=&(\sum_k a_k\ot a'_{k[0]})\ot_{B^e} m(a'_{k[1]}\cdot x)\\
&\equal{\equref{7.1.0}}&
(\sum_k a_k\ot a'_{k[0]})\ot_{B^e} \xi_M(a'_{k[1]}\cdot x)m\\
&=&(\sum_k a_k\xi_M(a'_{k[1]}\cdot x)\ot a'_{k[0]})\ot_{B^e} m\\
&\equal{\equref{1.2.9}}&
(\sum_k a_{k[1]}\cdot\xi_M(a'_{k[1]}\cdot x) a_{k[0]}\ot a'_{k[0]})\ot_{B^e} m\\
&=&(\sum_k a_{k[1]}\cdot\xi_M(S(a_{k[2]})\cdot x) a_{k[0]}\ot a'_{k})\ot_{B^e} m.
\end{eqnarray*}
Now take an arbitrary $n\in N$. Applying \leref{5.5}, we find
\begin{eqnarray*}
&&\hspace*{-2cm}
\sum_{i,k} \Bigl(\bigl(\xi_M(x)a_{k{[0]}}\ot l_i(a_{k{[1]}})\bigl)\ot_{B^e} m\Bigr)
\ot_B\Bigl(\bigl(r_i(a_{k{[1]}})\ot a'_k\bigl)\ot_{B^e} n\Bigr)\\
&=&
\sum_{i,k} \Bigl(\bigl((a_{k{[1]}}\cdot (\xi_M(S(a_{k{[2]}})\cdot x))
a_{k{[0]}}\ot l_i(a_{k{[3]}})\bigl)\ot_{B^e} m\Bigr)\\
&&\hspace*{2cm}
\ot_B\Bigl(\bigl(r_i(a_{k{[3]}})\ot a'_k\bigl)\ot_{B^e} n\Bigr).
\end{eqnarray*}
Now we apply
$$g^{-1}:\ (A^{\sq e}\ot_{B^e} M)\ot_B^H (A^{\sq e}\ot_{B^e} N)\to
A^{\sq e}\ot_{B^e} (M\ot_B N)$$
to both sides (see \equref{5.7.1}). Using \equref{1.2.5}, we obtain
\begin{eqnarray*}
&&\hspace*{-2cm}
\sum_k (\xi_M(x)a_k\ot a'_k)\ot_{B^e} (m\ot_B n)\\
&=&
(\sum_k a_{k[1]}\cdot\xi_M(S(a_{k[2]})\cdot x) a_{k[0]}\ot a'_{k})\ot_{B^e} (m\ot_B n).
\end{eqnarray*}
Now $M\ot_B N\cong B$. It follows that
$$\sum_k (\xi_M(x)a_k\ot a'_k)\ot_{B^e} b=
\left(\sum_k a_{k[1]}\cdot\xi_M(S(a_{k[2]})\cdot x) a_{k[0]}\ot
a'_{k}\right)\ot_{B^e} b,$$ for all $\sum_k a_k\ot a'_k\in A^{\sq
e}$, $x\in Z(B)$ and $b\in B$. Using \leref{7.2}, we find that
\begin{align*}
 \sum_k \xi_M(x)a_k\ot_B a'_k &=
\sum_k a_{k[1]}\cdot\xi_M(S(a_{k[2]})\cdot x) a_{k[0]}\ot_B a'_{k} \\
&= \sum_kS(a'_{k[1]})\cdot\xi_M(a'_{k[2]}\cdot x) a_{k}\ot_B
a'_{k[0]}
\end{align*}
for all $\sum_k a_k\ot_B a'_k\in A\ot_B^H A^{\rm op}$ and $x\in Z(B)$.\\
Now take $h\in H$. It follows from (\ref{eq:1.2.3}-\ref{eq:1.2.4}) that
$\gamma_A(h)=\sum_i l_i(h)\ot_B r_i(h)\in A\ot_B^HA^{\rm op}$. Therefore
\begin{eqnarray*}
&&\hspace*{-2cm}
\sum_i \xi_M(x)l_i(h)\ot_B r_i(h)\\
&=&
\sum_i (S(r_i(h)_{[1]})\cdot \xi_M(r_i(h)_{[1]}\cdot x))l_i(h)\ot_B r_i(h)_{[0]}\\
&\equal{\equref{1.2.3}}&
\sum_i S(h_{(2)})\cdot \xi_M(h_{(3)}\cdot x)\ot l_i(h_{(1)})\ot_B r_i(h_{(1)}).
\end{eqnarray*}
We apply $(A\ot\varepsilon)\circ \gamma_A$ to both sides; this gives
$$\xi_M(x)\varepsilon(h)=S(h_{(1)})\cdot \xi_M(h_{(2)}\cdot x),$$
and, finally,
$$h\cdot \xi_M(x)=h_{(1)}\cdot \xi_M(x)\varepsilon(h_{(2)})
=(h_{(1)}S(h_{(2)}))\cdot \xi_M(h_{(3)}\cdot x)=\xi_M(h\cdot x),
$$ which gives the desired formula.
\end{proof}

\begin{proposition}\prlabel{7.4}
The action of $\Pic(B)$ on $Z(B)$ induces an action of $\Pic(B)^H$ on
$Z^n(H,Z(B))$, $B^n(H,Z(B))$ and $H^n(H,Z(B))$. More precisely, if
$f:\ H^{\ot n}\to Z(B)$ is a cocycle (resp. a coboundary), then
$\xi_M\circ f$ is also a cocycle (resp. a coboundary).
\end{proposition}

\begin{proof}
This follows immediately from \prref{7.3} and the definition of Sweedler
cohomology, see \cite{Sweedler} or \cite[Sec. 9.1]{Caenepeel}.
\end{proof}

Since $\Pic(B)^H$ acts on $H^2(H,Z(B))$, we can consider the
cohomology group $H^1(\Pic(B)^H,H^2(H,Z(B)))$.

\begin{theorem}\thlabel{7.5}
$g_3\in Z^1(\Pic(B)^H,H^2(H,Z(B)))$.
\end{theorem}

\begin{proof}
Let $[\ul{M}], [\ul{M}']\in \Pic(B)^H$, and consider the
corresponding total integrals \[t_M:\ H\to E:={}_{A^{\sq
e}}\END(A^{\sq e}\ot_{B^e} M),\quad t_{M'}:\ H\to E'.\] We recall
from \seref{4} that $[\sigma_M]=g_3[\ul{M}]$ is defined by the
formula
$$t_M(k)\circ t_M(h)=\sigma_M(h_{(1)}\ot k_{(1)})t_M(h_{(2)}k_{(2)}).$$
This means that
\begin{eqnarray*}
&&\hspace*{-2cm}
(t_M(k)\circ t_M(h))(1_{A^{\sq e}}\ot_{B^e} m)\\
&=& t_M(k)((m(h)^+\ot m(h)^-)\ot_{B^e} m(h)^0)\\
&=& (m(h)^+m(h)^0(k)^+\ot m(h)^0(k)^-m(h)^-)\ot_{B^e} m(h)^0(k)^0
\end{eqnarray*}
equals
\begin{eqnarray*}
&&\hspace*{-1cm}
\sigma_M(h_{(1)}\ot k_{(1)})t_M(h_{(2)}k_{(2)}) (1_{A^{\sq e}}\ot_{B^e} m)\\
&=& (\sigma_M(h_{(1)}\ot k_{(1)}) m(h_{(2)}k_{(2)})^+ \ot m(h_{(2)}k_{(2)})^-)
\ot_{B^e}m(h_{(2)}k_{(2)})^0.
\end{eqnarray*}
Then we compute
\begin{eqnarray*}
&&\hspace*{-2cm}
\bigl(t_{M\ot_BM'}(k)\circ t_{M\ot_BM'}(h)\bigr)
(1_{A^{\sq e}}\ot_B (m\ot_B m'))\\
&\equal{\equref{5.9.2}}&
t_{M\ot_BM'}(k)\bigl(
(m(h_{(1)})^+\ot m'(h_{(2)})^-)\\
&&\hspace*{15mm}\ot_{B^e} (m(h_{(1)})^0m(h_{(1)})^-m'(h_{(2)})^+\ot_B
m'(h_{(2)})^0)\bigr)\\
&\equal{\equref{5.9.2}}&
\bigl(m(h_{(1)})^+m(h_{(1)})^0(k_{(1)})^+\ot m'(h_{(2)})^0(k_{(2)})^-m'(h_{(2)})^-\bigr)\\
&&\hspace*{15mm}\ot_{B^e}
\bigl(m(h_{(1)})^0(k_{(1)})^0m(h_{(1)})^0(k_{(1)})^- m(h_{(1)})^-\\
&&\hspace*{15mm} m'(h_{(2)})^+m'(h_{(2)})^0(k_{(2)})^+\ot_B
m'(h_{(2)})^0(k_{(2)})^0\bigr),
\end{eqnarray*}
hence
\begin{align*}
g&\Bigl(\bigl(t_{M\ot_BM'}(k)\circ t_{M\ot_BM'}(h)\bigr)
(1_{A^{\sq e}}\ot_B (m\ot_B m'))\Bigr)\\
&=
\bigl((m(h_{(1)})^+m(h_{(1)})^0(k_{(1)})^+\ot m(h_{(1)})^0(k_{(1)})^- m(h_{(1)})^-)\\
&\hspace*{15mm}
\ot_{B^e} m(h_{(1)})^0(k_{(1)})^0\bigr)\\
&\hspace*{15mm} \ot_B
\bigl((m'(h_{(2)})^+m'(h_{(2)})^0(k_{(2)})^+\ot m'(h_{(2)})^0(k_{(2)})^-m'(h_{(2)})^-)\\
&\hspace*{15mm}
\ot_{B^e} m'(h_{(2)})^0(k_{(2)})^0\bigr)\\
&= \bigl((\sigma(h_{(1)}\ot k_{(1)}) m(h_{(2)}k_{(2)})^+\ot
m(h_{(2)}k_{(2)})^-)
\ot_{B^e} m(h_{(2)}k_{(2)})^0\bigr)\\
&  \hspace*{15mm} \ot_B \bigl((\sigma'(h_{(3)}\ot k_{(3)})
m'(h_{(4)}k_{(4)})^+\ot m'(h_{(4)}k_{(4)})^-) \\
& \hspace*{15mm} \ot_{B^e} m'(h_{(4)}k_{(4)})^0\bigr)\\
&\equal{\equref{7.1.1}} \bigl((\sigma(h_{(1)}\ot k_{(1)})
\xi_M(\sigma'(h_{(2)}\ot k_{(2)}))
m(h_{(3)}k_{(3)})^+\ot m(h_{(3)}k_{(3)})^-)\\
& \hspace*{15mm} \ot_{B^e}m(h_{(3)}k_{(3)})^0\bigr) \ot_B \bigl((
m'(h_{(4)}k_{(4)})^+\ot m'(h_{(4)}k_{(4)})^-) \\
 &\hspace*{15mm} \ot_{B^e} m'(h_{(4)}k_{(4)})^0\bigr)\\
&=\sigma(h_{(1)}\ot k_{(1)}) \xi_M(\sigma'(h_{(2)}\ot k_{(2)}))
g\bigl( t_{M\ot_BM'}(hk)(1_{A^{\sq e}}\ot_B (m\ot_B m'))\bigr).
\end{align*}
This shows that
$$t_{M\ot_BM'}(k)\circ t_{M\ot_BM'}(h)=
\sigma(h_{(1)}\ot k_{(1)}) \xi_M(\sigma'(h_{(2)}\ot
k_{(2)}))t_{M\ot_BM'}(hk).$$ Consequently,
$$\sigma_{M\ot_B M'}=\sigma_M*(\xi_M\circ \sigma_{M'}),$$ which
proves the Theorem.
\end{proof}

\section{Galois objects over noncocommutative Hopf algebras}\selabel{9}
Let $H$ be a (possibly non-cocommutative) Hopf algebra with bijective
antipode, and $A$ an $H$-Galois extension of $B=A^{{\rm co}H}$.
We can still define the Picard groups $\Pic^H(A)$, $\Pic(B)$ and
$\Pic^{\sq_H}(B)$, and we still have that $\Pic^H(A)\cong \Pic^{\sq_H}(B)$,
cf. \seref{5}. We can therefore ask whether the exact sequence
from \thref{6.6} can be generalized to non-cocommutative Hopf algebras.
The obstructions are the following.
\begin{enumerate}
\item We need the property that $A\sq_H A^{\rm op}$ is an $H$-Galois extension
(see \thref{2.5} and \prref{2.7}) in order to apply \coref{4.13} (with
$H$ replaced by $A\sq_H A^{\rm op}$);
\item We used the fact that $H$ is cocommutative when we defined the
$H$-stable part of $\Pic(B)$ (see \seref{5b});
\item We want to have a group structure on $\Omega_{A\sq_H A^{\rm op}}$.
\end{enumerate}
These problems can be fixed in the case where the algebra of coinvariants $B$
coincides with the groundfield $k$, that is, when $A$ is a Galois object.
Examples of Galois objects are for example classical Galois field extensions
(then $H=(kG)^*$, with $G$ a finite group); other examples of Galois objects
over noncocommutative algebras have been studied in \cite{A1,A2}.\\
In this case, $\Omega_{A\sq_H A^{\rm op}}\cong \Alg(H,k)$ is a group, by \prref{4.3b},
and problem 3) is fixed. To handle problem 1), we invoke the theory of
Hopf-Bigalois objects, as developed by Schauenburg \cite{Schauenburg}.
If $A$ is a right $H$-Galois object, then there exists another Hopf algebra
$L=L(A,H)$, unique up to isomorphism, such that $A$ is an $(L,H)$-Bigalois object,
that is, $A$ is left $L$-Galois object, a right $H$-Galois object, and an
$(L,H)$-bicomodule. For the construction of $L$, we refer to \cite[Sec. 3]{Schauenburg}.
If $H$ is cocommutative, then $L=H$. We can then introduce the Harrison groupoid
\cite[Sec. 4]{Schauenburg}. Objects are Hopf algebras with bijective antipode,
morphisms are Hopf-Bigalois objects, and the composition of morphisms is given
by the cotensor product. The inverse of a morphism $A$ between $L$ and $H$
(that is, an $(L,H)$-Bigalois object) is $A^{\rm op}$, with left $H$-coaction $\lambda$ given
by the formula $\lambda(a)=S^{-1}(a_{[1]})\ot a_{[0]}$. In particular,
$(A\sq_H A^{\rm op})$ is an $(L,L)$-Bigalois object, and, in particular, a right
$H$-Galois object. Applying \prref{4.3b} and \coref{4.13}, we obtain
$$\ol{\Omega}_{A\sq_H A^{\rm op}}\cong \Omega_{A\sq_H A^{\rm op}} \cong
\Alg(A\sq_H A^{\rm op},k)\cong \Alg(L,k).$$
The isomorphism $\Alg(A\sq_H A^{\rm op},k)\cong \Alg(L,k)$ can also be obtained
as follows. Since $A^{\rm op}$ is the inverse of $A$ in the Harrison groupoid,
we have that $A\sq_H A^{\rm op}\cong L$ as bicomodule algebras.\\
Since $\Pic(B)=1$ ($k$ is a field), the map $\Pic^{\sq_H}(B)\to \Pic(B)$ is trivial.
Its kernel is $\ol{\Omega}_{A\sq_H A^{\rm op}}$, so we obtain the following result.

\begin{proposition}\prlabel{9.1}
Let $H$ be a Hopf algebra with bijective antipode, $A$ a right $H$-Galois object,
and $L=L(A,H)$. Then $\Pic^H(A)\cong \Pic^{\sq_H}(k)\cong \Alg(L,k)$.
\end{proposition}

If $H$ is cocommutative, then $L=H$, so $\Pic^H(A)\cong \Alg(H,k)$. This isomorphism
can be described explicitely. The isomorphism $\Alg(H,k)\to \Alg(A\sq_H A^{\rm op}, k)$
is a particular case of \equref{6.5.1}. For an algebra morphism $\alpha:\ H\to k$, the
corresponding $\phi_\alpha:\ A\sq_H A^{\rm op}\to k$ is given by
$$\phi_\alpha(\sum_j a_j\ot a'_j)=\sum_j a_ja'_{j[0]}\alpha(a'_{j[1]}),$$
and the corresponding $A\sq_H A^{\rm op}$-action on $k$ is induced by $\alpha$.\\
Let us now compute the corresponding $A$-bimodule structure on $A$.
It is shown in \cite[Prop. 2.3]{CCMT} that we have a right $H$-colinear isomorphism
$$f:\ A\ot (A\sq_H A^{\rm op})\to A\ot A^{\rm op},~~
f\bigl(a\ot (\sum_j a_j\ot a'_j)\bigr)= \sum_j aa_j\ot a'_j.$$
The inverse of $f$ is given by the formula
$$f^{-1}(a\ot a')=\sum_i l_i(S(a'_{[1]}))\ot r_i(S(a'_{[1]}))\ot a'_{[0]}.$$
For $N\in {}_{A\sq_HA^{\rm op}}\Mm$, we have an isomorphism
$$g:\ A\ot N\mapright{\psi} A\ot (A\sq_H A^{\rm op}) \ot_{A\sq_H A^{\rm op}} N
\mapright{f\ot N} (A\sq_HA^{\rm op})\ot_{A\sq_H A^{\rm op}} N.$$
Here $\psi$ is the natural isomorphism. The $A$-bimodule structure on $A\ot N$ is
obtained by transporting the $A$-bimodule structure on
$(A\sq_HA^{\rm op})\ot_{A\sq_H A^{\rm op}} N$ to $A\ot N$ using $g$. Take
$a,a',a''\in A$ and $n\in N$. Then
$$a'g(a\ot n)a''=a'((a\ot 1)\ot_{A\sq_H A^{\rm op}} n)a''=
(a'a\ot a'')\ot_{A\sq_H A^{\rm op}} n.$$
Now
\begin{eqnarray*}
&&\hspace*{-1cm}
a'\cdot (a\ot n)\cdot a''= g^{-1}\bigl(a'g(a\ot n)a''\bigr)
= \psi^{-1}\bigl( f^{-1}(a'a\ot a'')\ot_{A\sq_H A^{\rm op}} n\bigr)\\
&=&
\sum_i a'al_i(S(a''_{[1]}))\ot \bigl( r_i(S(a''_{[1]}))\ot a''_{[0]}\bigr)n \in A\ot N.
\end{eqnarray*}
Now let $N=k$, with left $A\sq_H A^{\rm op}$-action given by $\phi_\alpha$,
and identify $A\ot N\cong A$ using the natural isomorphism. The corresponding
$A$-bimodule structure on $A\ot N\cong A$ is then given by the formula
\begin{eqnarray*}
&&\hspace*{-1cm}
a'\cdot a \cdot a''=
\sum_i a'al_i(S(a''_{[1]})) \phi_\alpha\bigl( r_i(S(a''_{[1]}))\ot a''_{[0]}\bigr)\\
&=&
\sum_i a'al_i(S(a''_{[2]})) r_i(S(a''_{[2]})) a''_{[0]}\alpha(a''_{[1]})~
\equal{\equref{1.2.5}}~
a'aa''_{[0]}\alpha(a''_{[1]}).
\end{eqnarray*}
We conclude that the $(A\ot A^{\rm op},H)$-Hopf module $P$ representing the
element in $\Pic^H(A)$ corresponding to $\alpha$ is equal to $A$ as a left
$A$-module and a right $H$-comodule, and with right $A$-module action
given by the formula
\begin{equation}\eqlabel{9.1.1}
a\cdot a'=aa'_{[0]}\alpha(a'_{[1]}).
\end{equation}

\begin{example}\exlabel{9.2}
Let $q=p^d$, and $k$ a field of characteristic $p$. Consider the Hopf algebra
$H=k[x]/(x^q-x)$, with $x$ primitive and $S(x)=-x$. If $d=1$, then $H$ is the
dual of the group algebra over the cyclic group of order $p$.
The $H$-Galois are known,
see for example \cite[Sec. 11.3]{Caenepeel} for detail. More precisely, the group of
Galois objects $\Gal(k,H)\cong k/\{a^q-a~|~a\in k\}$. The Galois object corresponding
to $a\in k$ is the Artin-Schreier extension
$$S=k[y]/(y^q-y-a)$$
with coaction $\rho_S(y)= y\ot 1+1\ot x$. Furthermore
$$\Alg(H,k)\cong \{b\in k~|~b^q=b\}.$$
The algebra morphism $\alpha$ corresponding to $b\in k$ is determined by
the formula $\alpha(x)=b$. Now fix $a\in k$, and consider $S=k[y]/(y^q-y-a)$.
It follows from \prref{9.1} that
$$\Pic^H(S)\cong \{b\in k~|~b^q=b\}.$$
The $(S\ot S^{\rm op},H)$-Hopf module $P$ representing the element of
$\Pic^H(S)$ corresponding to $b$ satisfying $b^q=b$ is equal to $S$ as a left
$S$-module and a right $H$-comodule. The right $S$-action on $P$ is
completely determined by the right action of $y$ on $p\in P=S$. Since
$y_{[0]}\alpha(y_{[1]})=y+b$, formula \equref{9.1.1} takes the form
$$p\cdot y= p(y+b).$$
\end{example}

\begin{example}\exlabel{9.3}
We keep the notation of \exref{9.2}. Let $B$ be a $k$-algebra, and
$A= B\ot S$, $\rho_A=B\ot \rho:\ B\ot S\to B\ot S\ot H$. Then
$$\can_A=B\ot \can_S:\ A\ot_B A=(B\ot S)\ot_B (B\ot S)\cong B\ot S\ot S
\to B\ot S\ot H=A\ot H$$
is an isomorphism, hence $A$ is an $H$-Galois extension of $B$.\\
We claim that the Miyashita-Ulbrich action on $Z(B)$ is trivial. Let 
$\gamma_S(h)=\sum_i l_i(h)\ot r_i(h)\in S\ot S$, for all $h\in H$. It is
easy to see that
$$\can_A(\sum_i 1_B\ot l_i(h)\ot r_i(h))=1_B\ot 1_S\ot h=1_A\ot h,$$
hence
$$\gamma_A(h)=\sum_i (1_B\ot l_i(h))\ot_B (1_B\ot r_i(h)),$$
and, for $x\in Z(B)\cong Z(B)\ot k$,
$$h\cdot x=\sum_i (1_B\ot l_i(h))(1_B\ot 1_k) (1_B\ot r_i(h))=\varepsilon(h)x.$$
Now it follows that
$$H^1(H,Z(B))\cong \Alg(H,B)=\{b\in B~|~b^q=b\}.$$
Our next aim is to show that every element of $\Pic(B)$ is $H$-stable.
First observe that $A^{\rm op}=B^{\rm op}\ot S^{\rm op}$, with $S^{\rm op}=S$
as an algebra, and with $H$-coaction given by
$\rho(y)=y\ot 1-1\ot x$. Then
$$A\sq_H A^{\rm op}=B\ot B^{\rm op}\ot (S\sq_H S^{\rm op})=B^e\ot S^{\sq e}.$$
Now let $M\in \dul{\Pic}(B)$. Then $A^{\sq e}\ot_{B^e} M= M\ot S^{\sq e}
\cong M\ot H$, since $S^{\sq e}\cong H$. This shows that $M$ is $H$-stable,
and it follows that $\Pic(B)=\Pic(B)^H$. The exact sequence from \thref{6.6}
specializes to
$$1\to \{b\in B~|~b^q=b\}\to \Pic^H(A)\to \Pic(B)\to H^2(H,Z(B)).$$
\end{example}

Before we present our final \exref{9.4}, we make the following observation.
Suppose that $H$ is a finite dimensional commutative Hopf algebra. Then $H^*$
is a cocommutative Hopf algebra. If $A$ is an $H^*$-Galois object, then
$A$ is an $H$-module algebra, with left $H$-action $h(a)=\lan a_{[1]},h\ran a_{[0]}$.
Furthermore $\Alg(H^*,k)=G(H)$, the group of grouplike elements of $H$.
Take $g\in G(H)$; \equref{9.1.1} can then be rewritten as
\begin{equation}\eqlabel{9.4.1}
a\cdot a'=ag(a').
\end{equation}

\begin{example}\exlabel{9.4}
In \cite{HP}, forms of the cyclic group algebra have been studied.
One of the examples is the following quotient of the trigonometric Hopf algebra
over $\QQ$:
$$H=\QQ[c,s]/(c^2+s^2-1,sc).$$
$H$ is a form of the group algebra over the cyclic group of order 4, that is,
$H\ot_{\QQ}\CC\cong \CC C_4$. The grouplike elements of
$H\ot_{\QQ}\CC= \CC[c,s]/(c^2+s^2-1,sc)$ are
$g_i=(c+is)^i$, $i=0,\cdots, 3$. It is easy to see that
$g_1,\, g_3\not\in H$ and $g_0=1,\, g_2=c^2-s^2\in H$, hence
$$G(H)=\{1,g_2=c^2-s^2\}.$$
An example of an $H^*$-Galois object is given in \cite[Remark p. 135]{HP}:
$A=\QQ(\mu)$, with $\mu=\sqrt[4]{2}$, and $H$-action given by the formulas
$$\begin{array}{cccc}
c(1)=1&c(\mu)=0&c(\mu^2)=-\mu^2&c(\mu^3)=0\\
s(1)=0&s(\mu)=-\mu&s(\mu^2)=0&s(\mu^3)=\mu^3
\end{array}$$
Since $G(H)=\{1,g_2\}$, it follows from \prref{9.1} that $\Pic^H(A)$ is the
cyclic group of order 2. Using \equref{9.4.1}, we can describe its nontrivial
element $[P]$. First observe that the action of $g_2$ on $A$ is given by the formula
$g_2(\mu^i)=(-1)^i\mu$. Then $P=A$ as a left $A$-module and a left $H$-module,
with right $A$-action given by
$$a\cdot \mu^i=(-1)^i \mu^i a.$$

\end{example}

\end{document}